\documentclass[11pt,reqno]{amsart}
\usepackage{amsmath}
\usepackage{amsfonts}

\numberwithin{equation}{section}
\usepackage{times}
\usepackage{amsmath,amsfonts,amstext,amssymb,amsbsy,amsopn,amsthm,eucal}
\usepackage{txfonts}
\usepackage{dsfont}
\usepackage{graphicx}   
\usepackage{hyperref}
\usepackage{accents}
\usepackage{enumerate}
\usepackage{xcolor}

\usepackage{verbatim}

\newcommand{\Lip}{\text{Lip}}

\newcommand{\Ric}{{\rm Ric}}

\newcommand{\Vol}{{\rm Vol}}

\newcommand{\Div}{{\rm div}}

\newcommand{\cH}{\mathcal{H}}

\newtheorem{theorem}{Theorem}[section]

\newtheorem{proposition}[theorem]{Proposition}
\newtheorem{lemma}[theorem]{Lemma}
\newtheorem{corollary}[theorem]{Corollary}
\theoremstyle{definition}
\newtheorem{definition}[theorem]{Definition}
\theoremstyle{remark}
\newtheorem{remark}{Remark}[section]
\theoremstyle{remark}

\theoremstyle{remark}

\theoremstyle{remark}
\theoremstyle{remark}
\DeclareMathOperator*\lowlim{\underline{lim}}
\DeclareMathOperator*\uplim{\overline{lim}}

\begin{document}

\title{Removable singularity of positive mass theorem with continuous metrics } 
\date{\today}

\author{Wenshuai Jiang}
\address{W. Jiang,~~School of mathematical Sciences,  Zhejiang University and School of Mathematics and Statistics, University of Sydney}
\email{wsjiang@zju.edu.cn}
\author{Weimin Sheng}
\address{W. Sheng,~~School of mathematical Sciences,  Zhejiang University}
\email{weimins@zju.edu.cn}
\author{Huaiyu Zhang}
\address{H. Zhang,~~School of mathematical Sciences,  Zhejiang University}
\email{huaiyuzhang@zju.edu.cn}
\thanks{The first author was supported by NSFC (No. 11701507, No. 12071425) and ARC DECRA190101471; The second author was supported by NSFC (No. 11971424, No. 12031017); The third author was supported by NSFC (No. 11971424)}

\begin{abstract}
	In this paper, we consider asymptotically flat Riemannnian manifolds $(M^n,g)$ with $C^0$ metric $g$ and $g$ is smooth away from a closed bounded subset $\Sigma$ and the scalar curvature $R_g\ge 0$ on $M\setminus \Sigma$. For given $n\le  p\le \infty$, if $g\in C^0\cap W^{1,p}$  and  the Hausdorff measure $\cH^{n-\frac{p}{p-1}}(\Sigma)<\infty$ when $n\le p<\infty$ or $\cH^{n-1}(\Sigma)=0$ when $p=\infty$, then we prove that the ADM mass of each end is nonnegative. Furthermore, if the ADM mass of some end is zero, then we prove that $(M^n,g)$ is isometric to the Euclidean space by showing the manifold has nonnegative Ricci curvature in RCD sense. This extends the result of \cite{Lee2015a} from spin to non-spin, also improves the result of Shi-Tam \cite{ShiTam} and Lee \cite{Lee13}.  Moreover, for $p=\infty$, this confirms a conjecture of Lee \cite{Lee13}.
	\end{abstract}
\maketitle
\tableofcontents


\section{Introduction}

One of the most famous results on manifolds with nonnegative scalar curvature is the positive mass theorem proved by Scheon and Yau \cite{Schoen1979,ScYa2,ScYa3}. They proved that the Arnowitt-Deser-Misner (ADM) mass of each end of an $n$-dimensional asymptotically flat manifold with nonnegative scalar curvature is nonnegative. Moreover, if the ADM mass of an end is zero, then the manifold is isometric to the Euclidean space. Later, this was used by Schoen \cite{Sc} to completely solve the Yamabe problem. If the manifold is spin, Witten \cite{Witten} proved the positive mass theorem by a different method (see also Parker and Taubes \cite{PaTa}, Bartnik \cite{Ba}). All the results are assumed to be smooth. It is natural to ask:

{\it If the manifold admits singularity in a subset $\Sigma$,  what is the conditions of $\Sigma$ such that the positive mass theorem still holds? }

The present paper is motivated by this question. It is necessary to assume that the metric is continuous. In \cite[Proposition 2.3]{ShiTam}, Shi and Tam has constructed an asymptotically flat metric with a cone singularity and nonnegative scalar curvature, but with negative ADM mass.

 Actually, there are many results on positive mass theorem for such problems with nonsmooth metrics by Miao \cite{Miao2002}, Shi-Tam \cite{ShiTam1}, \cite{ShiTam}, McFeron-Sz\'ekelyhidi \cite{McSz}, Lee-LeFloch \cite{Lee2015a}, and Li-Mantoulidis \cite{LiMa}.  Miao \cite{Miao2002} and  Shi-Tam \cite{ShiTam1} studied and proved a positive mass theorems with Lipschitz metric  and $\Sigma$ is a hypersurface satisfying certain conditions on the mean curvatures of the hypersurface $\Sigma$. \cite{ShiTam1} used this result to prove the positivity of the Brown-York quasilocal masss. McFeron and Sz\'ekelyhidi \cite{McSz} used Ricci flow giving a new proof of positivity of Miao's result and proved also the rigidity when the ADM mass is zero. Lee \cite{Lee13} considered a positive mass theorem for $(M^n,g)$ with bounded $W^{1,p}$-metric for $n<p\le \infty$ and are smooth away from a singular set $\Sigma$ with $\frac{n}{2}(1-n/p)$-dimensional Minkowski content vanishing. Lee's result was improved recently by Shi-Tam \cite{ShiTam}, where they showed that the singular set $\Sigma$ only requires  $(n-2)$-dimensional Minkowski content vanishing. If $(M^n,g)$ is spin, Lee and LeFloch \cite{Lee2015a} were able to prove a positive mass theorem for continuous metric with bounded $W^{1,n}$-norm, where the metric could be singular. Lee-LeFloch's theorem can be applied to all previous results for nonsmooth metrics under the additional assumption that the manifold is spin. If we assume further the information on the second derivative of metrics, the positive mass theorem was considered by Bartnik \cite{Ba} for metric with bounded $W^{2,p}$-norm with $p>n$.  Grant and Tassotti \cite{LiMa} considered continuous metric with bounded $W^{2,n/2}$-norm. Recently,  Li-Mantoulidis \cite{LiMa} considered bounded metric with skeleton singularities along a codimensional two submanifolds. Furthermore, it is very interesting that without further derivative  assumptions on metric Li-Mantoulidis \cite{LiMa} were able to prove a positive mass theorem in dimensional three with isolated singularity.

The main result of this paper can be considered as an extension of Shi-Tam \cite{ShiTam}, Lee \cite{Lee13} and Lee-LeFloch \cite{Lee2015a}. We improve and recover some of their results.  The main theorem we will prove is the following. Our proof would depend on the work of \cite{Miao2002,ShiTam,Lee13,LiMa} with some extension. 
\begin{theorem}\label{thm1.2}
	Let $M^n$ ($n\geq3$) be a smooth manifold. Let $g\in C^0\cap W^{1,p}_{loc}(M)$ ($n\le p\le \infty$) be a complete asymptotically flat metric on $M$. Assume $g$ is smooth away from a bounded closed subset $\Sigma$ with $\cH^{n-\frac{p}{p-1}}(\Sigma)<\infty$ if $n\le p<\infty$ or $\cH^{n-1}(\Sigma)=0$ if $p=\infty$, assume that $R_g\ge 0$ on $M\setminus \Sigma$. Then the ADM mass of $g$ of each end is nonnegative. Moreover, the ADM mass of one end is zero if and only if when $(M, g)$ is isometric to Euclidean space.
\end{theorem}


\begin{remark}
	For the rigidity part, we will show such space has nonnegative Ricci curvature in RCD sense provided that the mass is zero. The rigidity would follow by the volume rigidity of nonnegative Ricci curvature 
\end{remark}
\begin{remark}
	The assumption of continuity of the metric is necessary, see Proposition 2.3 of \cite{ShiTam} for an example.
\end{remark}
\begin{remark}
	For the case $p=\infty$, the condition $\cH^{n-1}(\Sigma)=0$ is optimal, since one can construct counterexample if $\Sigma$ is a hypersurface. In particular, this confirms a conjecture of Lee \cite{Lee13}.
\end{remark}

\vskip 3mm
\noindent
\textbf{Organization:} 
In section 2, we will recall some basic results of asymptotically flat manifold. We will use smooth metric approximating the singular metric in $W^{1,p}$ sense and prove that the scalar curvature is closed in the integral sense (see Lemma \ref{lm2.2}). Furthermore, we will check in Lemma \ref{lm4.1} that the singular metric has nonnegative scalar curvature in distribution sense as introduced in \cite{Lee2015a}.

\noindent
In section 3, we will prove the mass is nonnegative. The proof is based on the idea from \cite{Miao2002} and our approximation Lemma \ref{lm2.2} and some uniform estimates for the conformal functions.

\noindent
In section 4, we will prove the rigidity part when mass is zero. First, based on the argument of \cite{ShiTam} we can show that the metric is Ricci flat away from the singular set $\Sigma$.  Next, based on a gradient estimate for harmonic function we will show that the space has nonnegative Ricci curvature in RCD sense. The rigidity result follows directly by the volume rigidity of nonnegative Ricci curvature. 

\section{Background}
In this section we will recall some definitions and study the smoothing approximation of metric. We will also prove the weak nonnegative scalar curvature under the same condition as Theorem \ref{thm1.2}. Let us give some fundamental definitions firstly.

\begin{definition}[asymptotically flat]\label{definition1}
Let $M$ be a smooth $n$-manifold, and $g$ be a $C^0$ metric on $M$. We say that $(M,g)$ is asymptotically flat if there exists a compact subset $K$ on $M$, such that $g$ is $C^2$ on $M\backslash K$, $M\backslash K$ has finite many components, says $\Sigma_l$, $l=1,\dots,p$, and for each component $\Sigma_l$ there exists a smooth diffeomorphism $\Phi_l$ from it to $\mathbb{R}^n$ minus a ball, such that if we see $\Phi_l$ as a coordinate system on $\Sigma_l$, then \[g_{ij}-\delta_{ij}=O(|x|^{-\delta})\] \[g_{ij,k}=O(|x|^{-\delta-1})\] \[g_{ij,kl}=O(|x|^{-\delta-2}),\] where $\delta$ is some constant greater than $(n-2)/2$, and the commas denote partial derivatives in the coordinate system. We also call each components $\Sigma_l$ as an end of $M$.
\end{definition}
\begin{definition}[ADM mass]Given a asymptotically flat manifold $(M,g)$, we define the ADM mass of each end $\Sigma_l$ as the limit \[\lim\limits_{r \to \infty }\frac{1}{2(n-1)}\omega_{n-1}\int_{S_r}\sum_{i,j=1}^n(g_{ij,i}-g_{ii,j})\nu^jd\mu,\] 
where $S_r$ is the coordinate sphere in $(\Sigma_l,\Phi_l)$ of radius $r$, $\nu$ is the unit outward normal vector of $S_r$, $\omega_{n-1}$ is the volume of the unit $(n-1)$ dimensional sphere and $d\mu$ is the volume form of $S_r$ in Euclidean metric. In a given end $\Sigma_l$, we will denote $r=\Phi_l^*\left(\sqrt{\Sigma_{i=1}^n (x^i)^2}\right)$ in this paper.
\end{definition}
Bartnik \cite{Ba} proved that this limit exists provided that the scalar curvature of $g$ is integrable on M  and it is a geometric invariant. We can denote the mass of the end $\Sigma_i$ by $m(\Sigma_i,g)$. In this article, we only need to check one end of M, and then the case of other ends holds in the same way, so we just choose an arbitrary end and denote the mass by $m(g)$ for simplicity. 

In this paper, we will fix a smooth background metric $h$ on $M$ such that $C^{-1}h\le g\le Ch$ for some constant $C>1$, and $h=g$ outside some compact subset of $M$. All the convergencs of functions or tensors are taken with respect to $h$. And we will let $\tilde\nabla$ denote the covariant derivative taken with respect to $h$.

\begin{definition}
Let $M^n$ be a smooth manifold and $g$ be a $C^0\cap W^{1,p}_{\rm{loc}}$ metric on $M$ and $h$ is a fixed smooth metric as above, for a family of smooth function $f_\delta$, we say that $f_\delta$ converge to a function $f$ locally in $W^{1,p}$-norm, if for any $\epsilon>0$ and for any $0<r<1$, there exists some $\delta_0>0$, such that for any $\delta\in(0,\delta_0),x\in M$, we have
\begin{align*}
\int_{B_r(x)}| f_\delta-  f|^p d\mu_h<\epsilon, \int_{B_r(x)}|\tilde\nabla f_\delta-\tilde\nabla f|^p d\mu_h<\epsilon,
\end{align*}
here and below the norm are taken with respect to $h$.

We say that $f_\delta$ converge to a function $f$ locally in $C^{0}$-norm, if for any $\epsilon>0$ and for any $r>0$, there exists some $\delta_0>0$, such that for any $\delta\in(0,\delta_0),x\in M$, we have
\begin{align*}
\sup_{B_r(x)}|f_\delta-f|d\mu_h<\epsilon,
\end{align*}

For a family of smooth tensors $T_\delta$, $\delta>0$, we say that $T_\delta$ converge to $T$ locally in $W^{1,p}$-norm or in $C^0$-norm, if $T_\delta$ and $T$ are of the same type, and for each chart on $M$, the component functions $(T_\delta)_{ijk\ldots}^{abc\ldots}$converge to $T_{ijk\ldots}^{abc\ldots}$ respectively locally in $W^{1,p}$-norm or in $C^0$-norm .
\end{definition}

\begin{definition}
Let $(M^n,g)$ be $C^0\cap W^{1,n}_{loc}$ metric and $h$ is a fixed smooth metric as above. We can define the scalar curvature distribution as in \cite{Lee2015a}
\begin{align}
	\langle R_g,\varphi\rangle :=\int_M \left(-V\cdot \tilde\nabla \left(\varphi \frac{d\mu_g}{d\mu_h}\right)+F\varphi \frac{d\mu_g}{d\mu_h}\right)d\mu_h
\end{align}
for any compactly supported $\varphi\in W^{1,n/(n-1)}$,
where the dot product are taken with respect to $h$, $V$ is a vector field and $F$ is a scalar field,
\begin{align}
\Gamma_{ij}^k&:=\frac{1}{2}g^{kl}\left(\tilde\nabla_i g_{jl}+\tilde\nabla_j g_{il}-\tilde\nabla_l g_{ij}\right), \\
	V^k&:=g^{ij}\Gamma_{ij}^k-g^{ik}\Gamma_{ji}^j=g^{ij}g^{k\ell}(\tilde\nabla_jg_{i\ell}-\tilde\nabla_{\ell}g_{ij}),\\
	F&:=\bar R-\tilde \nabla_kg^{ij}\Gamma_{ij}^k+\tilde\nabla_kg^{ik}\Gamma_{ji}^i+g^{ij}\left(\Gamma_{k\ell}^k\Gamma_{ij}^\ell-\Gamma_{j\ell}^k\Gamma_{ik}^\ell\right),
\end{align}
and $\mu_h$ is the Lebesgue measure with respect to $h$. By \cite{Lee2015a}, $\langle R_g,\varphi\rangle$ is independent of $h$, and it coincides with the integral $\int_M R_g \varphi d\mu_h$ when $g$ is $C^2$ and $R_g$ is defined in classical sense. We say that $g$ has weakly nonnegative scalar curvature if for any nonnegative test function $\varphi$, we have $\langle R_g,\varphi\rangle\ge 0$.
\end{definition}

\subsection{Smoothing the metric}
The following mollification lemma could be found in \cite[Lemma 4.1]{grant2014positive}, though their lemma is a $W^{2,\frac{n}{2}}$ version, our version could be proved in the same manner.
\begin{lemma}[\cite{grant2014positive}]\label{lm2.1}
Let $M^n$ be a smooth manifold and $g$ be a $C^0\cap W^{1,n}_{\rm{loc}}$ metric on $M$, then there exists a family of smooth metric $g_\delta$, $\delta>0$, such that $g_\delta$ converge to $g$ locally both in $C^0$-norm and in $W^{1,n}$ norm. Moreover, if $g$ is smooth away from a compact subset, then we can take $g_\delta$ such that $g_\delta$ coincide with $g$ outside some compact set $K$ independent of $\delta$.  
\end{lemma}
\begin{remark}
Since we take $g_\delta$ such that $g_\delta$ coincide with $g$ outside some compact set $K$ independent of $\delta$, by the definition of ADM mass we have that $m(g_\delta)=m(g)$.
\end{remark}
For this mollification, the scalar curvature distribution has an approximation. Concretely, we have the following lemma
\begin{lemma}\label{lm2.2}
	Let $M^n$ be a smooth manifold and $g$ be a $C^0\cap W^{1,n}_{\rm{loc}}$ metric on $M$. Suppose that the $L^2$ Sobolev constant of $(M,g)$ has an upper bound $C_s$. Let $g_\delta$ be the mollification in Lemma \ref{lm2.1}. Suppose that $g_\delta$ coincide with $g$ outside some compact set $K$.  Then  
\[|\langle R_{g_\delta},u^2\rangle-\langle R_g,u^2\rangle|\leq \Psi(\delta)\int_M|\nabla u|^2d\mu_{g },\forall u\in C_0^\infty(M),\]
where $R_{g_\delta}$ is scalar curvature of $g_\delta$ and $\lim_{\delta\to 0}\Psi(\delta)=0$. Here $\Psi(\delta)$ depends only on the Sobolev constant $C_s$ and the $W^{1,n}$-norm of $|g-g_\delta|$. 
\end{lemma}
\begin{remark}
The $L^2$ Sobolev constant condition means that for any $u\in C_0^\infty(M)$, we have
\[\left(\int_M u^{\frac{2n}{n-2}}d\mu_g\right)^{\frac{n-2}{n}}\leq C_s \int_M|\nabla u|^2d\mu_g\]
holds. This always holds for asymptotically flat manifolds. 
\end{remark}
\begin{proof}
Let $h$ be a smooth metric on $M$ with $C^{-1}h<g<Ch$, here and below $C$ will denote some positive constant might depend on $n$ and $C_s$, but independent of $\delta$ and can vary from line to line. Then the Sobolev inequality also holds for metric $h$ since $h$ and $g$ are equivalent. Let $V_\delta$ and $F_\delta$ be the vector field and scalar field  in the definition of scalar curvature distribution of $g_\delta$. Then we have 
\begin{align}
	\lim_{\delta\to0^{+}}\left(\int_M|V^k_\delta-V^k|^nd\mu_h+\int_M|F_\delta-F|^{n/2}d\mu_h\right)=0,
\end{align}	
where the integrals are taken only in a compact subset due to the fact that $g=g_\delta$ away from a compact subset. 
Recall that for any $\varphi\in C_0^\infty$
\begin{align}
	\langle R_g,\varphi\rangle &=\int_M \left(-V\cdot \tilde\nabla \left(\varphi \frac{d\mu_g}{d\mu_h}\right)+F\varphi \frac{d\mu_g}{d\mu_h}\right)d\mu_h\\
	\langle R_{g_\delta},\varphi\rangle &=\int_M \left(-V_\delta\cdot \tilde\nabla \left(\varphi \frac{d\mu_{g_\delta}}{d\mu_h}\right)+F_\delta\varphi \frac{d\mu_{g_\delta}}{d\mu_h}\right)d\mu_h
\end{align}

Noting that $g_\delta$ coincides with $g$ away from a compact subset, the formula $\langle R_g,\varphi\rangle-\langle R_{g_\delta},\varphi\rangle$ only depends on the information in a compact subset. 
For the term involving $F_\delta$, by H\"older inequality and Sobolev inequality, we can estimate 
\begin{align}
	&\Big|\int_M F_\delta u^2 \frac{d\mu_{g_\delta}}{d\mu_h}d\mu_h-\int_M Fu^2\frac{d\mu_{g}}{d\mu_h}d\mu_h\Big|\\
	&\le \int_M \Big| F_\delta u^2\frac{d\mu_{g_\delta}}{d\mu_h}-Fu^2\frac{d\mu_{g}}{d\mu_h}\Big|d\mu_h\\
	&\le \int_M|F_\delta u^2-Fu^2|\frac{d\mu_{g_\delta}}{d\mu_h} d\mu_h+\int_M |F|u^2\Big| \frac{d\mu_{g_\delta}}{d\mu_h}-\frac{d\mu_{g}}{d\mu_h}\Big|d\mu_h\\
	&\le C \int_M|F_\delta u^2-Fu^2| d\mu_h+\sup_M\Big| \frac{d\mu_{g_\delta}}{d\mu_h}-\frac{d\mu_{g}}{d\mu_h}\Big| \int_M |F|u^2d\mu_h\\
&\le C\left(\int_M |F_\delta-F|^{n/2}d\mu_h \right)^{2/n}\left(\int_Mu^{2n/(n-2)}d\mu_h\right)^{(n-2)/n}+\sup_M\Big| \frac{d\mu_{g_\delta}}{d\mu_h}-\frac{d\mu_{g}}{d\mu_h}\Big| \left(\int_M |F|^{n/2}d\mu_h\right)^{2/n}\left(\int_M |u|^{2n/(n-2)}d\mu_h\right)^{(n-2)/n}\\
&\le \left(C\left(\int_M |F_\delta-F|^{n/2}d\mu_h \right)^{2/n}+C\sup_M\Big| \frac{d\mu_{g_\delta}}{d\mu_h}-\frac{d\mu_{g}}{d\mu_h}\Big| \left(\int_M |F|^{n/2}d\mu_h\right)^{2/n}\right) \int_M |\tilde\nabla u|^{2}d\mu_h \le \Psi(\delta) \int_M |\tilde\nabla u|^{2}d\mu_h, 
\end{align}	
where $\lim_{\delta\to 0}\Psi(\delta)=0$.
	For the term involving $V^k$, we can estimate 
	\begin{align}
		\Big|\int_M& V\cdot \tilde\nabla \left(u^2\frac{d\mu_g}{d\mu_h}\right)d\mu_h-\int_M V_\delta\cdot \tilde\nabla \left(u^2\frac{d\mu_{g_\delta}}{d\mu_h}\right)d\mu_h\Big|\\
		\le & \int_M |V-V_\delta|\cdot\Big|\tilde\nabla \left(u^2\frac{d\mu_{g_\delta}}{d\mu_h}\right)\Big|d\mu_h+\int_M |V|\cdot\Big|\tilde\nabla \left(u^2\frac{d\mu_{g}}{d\mu_h}-u^2\frac{d\mu_{g_\delta}}{d\mu_h}\right)\Big|d\mu_h\\
		\le & \left(\int_M|V-V_\delta|^nd\mu_h\right)^{1/n}\left(\int_M\Big|\tilde\nabla \left(u^2\frac{d\mu_{g_\delta}}{d\mu_h}\right)\Big|^{n/(n-1)}d\mu_h\right)^{(n-1)/n}\\
		&+ \left(\int_M|V|^nd\mu_h\right)^{1/n}\left(\int_M\Big|\tilde\nabla \left(u^2\frac{d\mu_{g_\delta}}{d\mu_h}-u^2\frac{d\mu_{g}}{d\mu_h}\right)\Big|^{n/(n-1)}d\mu_h \right)^{(n-1)/n}.
	\end{align}
	Notice that 
	\begin{align}
		\int_M&\Big|\tilde\nabla \left(u^2\frac{d\mu_{g_\delta}}{d\mu_h}\right)\Big|^{n/(n-1)}d\mu_h\\
		\le &C(n)\int_M\left( |\tilde\nabla u|\cdot |u|\frac{d\mu_{g_\delta}}{d\mu_h}\right)^{n/(n-1)}d\mu_h+C(n)\int_M\left( u^2|\tilde\nabla \frac{d\mu_{g_\delta}}{d\mu_h}|\right)^{n/(n-1)}d\mu_h\\
		\le & C\int_M |\tilde\nabla u|^{n/(n-1)}|u|^{n/(n-1)}d\mu_h+C(n)\left(\int_M u^{2n/(n-2)}d\mu_h\right)^{(n-2)/(n-1)}\left(\int_M|\tilde\nabla \frac{d\mu_{g_\delta}}{d\mu_h}|^nd\mu_h\right)^{1/(n-1)}\\
		\le &C\left(\int_M |\tilde\nabla u|^{2}d\mu_h\right)^{\frac{n}{2(n-1)}}\left(\int_M|u|^{2n/(n-2)}d\mu_h\right)^{\frac{n-2}{2(n-1)}}+C(n)\left(\int_M u^{2n/(n-2)}d\mu_h\right)^{(n-2)/(n-1)}\left(\int_M|\tilde\nabla \frac{d\mu_{g_\delta}}{d\mu_h}|^nd\mu_h\right)^{1/(n-1)}\\
		\le &C\left(1+\int_M|\tilde\nabla \frac{d\mu_{g_\delta}}{d\mu_h}|^nd\mu_h\right)^{1/(n-1)}\left(\int_M |\tilde\nabla u|^{2}d\mu_h\right)^{\frac{n}{(n-1)}}
	\end{align}
	and similarly, 
	\begin{align}
	\int_M&\Big|\tilde\nabla \left(u^2\frac{d\mu_{g_\delta}}{d\mu_h}-u^2\frac{d\mu_{g}}{d\mu_h}\right)\Big|^{n/(n-1)}d\mu_h\\
	\le & C\sup_M 	\Big|\frac{d\mu_{g_\delta}}{d\mu_h}- \frac{d\mu_{g}}{d\mu_h}\Big|^{n/(n-1)}\int_M |\tilde\nabla u|^{n/(n-1)}|u|^{n/(n-1)}d\mu_h\\
	&+ C \left(\int_M u^{2n/(n-2)}d\mu_h\right)^{(n-2)/(n-1)}\left(\int_M|\tilde\nabla \left(\frac{d\mu_{g_\delta}}{d\mu_h}-\frac{d\mu_{g}}{d\mu_h}\right)|^nd\mu_h\right)^{1/(n-1)}\\
	\le & C \sup_M 	\Big|\frac{d\mu_{g_\delta}}{d\mu_h}- \frac{d\mu_{g}}{d\mu_h}\Big|^{n/(n-1)}\left(\int_M |\tilde\nabla u|^{2}d\mu_h\right)^{\frac{n}{2(n-1)}}\left(\int_M|u|^{2n/(n-2)}d\mu_h\right)^{\frac{n-2}{2(n-1)}}\\
	&+ C \left(\int_M u^{2n/(n-2)}d\mu_h\right)^{(n-2)/(n-1)}\left(\int_M|\tilde\nabla \left(\frac{d\mu_{g_\delta}}{d\mu_h}-\frac{d\mu_{g}}{d\mu_h}\right)|^nd\mu_h\right)^{1/(n-1)}\\
	\le &C\left(\sup_M 	\Big|\frac{d\mu_{g_\delta}}{d\mu_h}- \frac{d\mu_{g}}{d\mu_h}\Big|^{n/(n-1)}
	+ \left(\int_M|\tilde\nabla \left(\frac{d\mu_{g_\delta}}{d\mu_h}-\frac{d\mu_{g}}{d\mu_h}\right)|^nd\mu_h\right)^{1/(n-1)}\right)\left(\int_M |\tilde\nabla u|^{2}d\mu_h\right)^{\frac{n}{(n-1)}}
	\end{align}
	Combining the estimates above,  we arrive at
	\begin{align}
		\Big|\int_M& V\cdot \tilde\nabla \left(u^2\frac{d\mu_g}{d\mu_h}\right)d\mu_h-\int_M V_\delta\cdot \tilde\nabla \left(u^2\frac{d\mu_{g_\delta}}{d\mu_h}\right)d\mu_h\Big|\\
		\le & C\left(\int_M|V_\delta-V|^nd\mu_h\right)^{1/n}\left( 1+\left(\int_M|\tilde\nabla \frac{d\mu_{g_\delta}}{d\mu_h}|^nd\mu_h\right)^{1/n}\right)\int_M |\tilde\nabla u|^{2}d\mu_h\\
		&+ C\left( \sup_M 	\Big|\frac{d\mu_{g_\delta}}{d\mu_h}- \frac{d\mu_{g}}{d\mu_h}\Big| 
	+  \left(\int_M|\tilde\nabla \left(\frac{d\mu_{g_\delta}}{d\mu_h}-\frac{d\mu_{g}}{d\mu_h}\right)|^nd\mu_h\right)^{1/n}\right)\int_M |\tilde\nabla u|^{2}d\mu_h.	
	\end{align}
	Therefore, we get 
	\begin{align}
		|\langle R_{g_\delta},u^2\rangle-\langle R_g,u^2\rangle|\le \Psi(\delta)\int_M |\tilde\nabla u|^{2}d\mu_h,\forall u\in C_0^\infty(M).
	\end{align}
	Since $g$ and $h$ is comparable, we get 
	\begin{align}
		|\langle R_{g_\delta},u^2\rangle-\langle R_g,u^2\rangle|\le \Psi(\delta)\int_M | \nabla u|^{2}d\mu_g,\forall u\in C_0^\infty(M).
	\end{align}
	which completes the proof of the lemma.
\end{proof}
\begin{remark}
In the same condition of Lemma \ref{lm2.2}, we can calculate in the same manner and get that 
\[|\langle R_{g_\delta},u\rangle-\langle R_g,u\rangle|\leq \Psi(\delta)\left(\int_M|\nabla u|^\frac{n}{n-1}d\mu_{g_\delta}\right)^\frac{n-1}{n},\forall u\in C_0^\infty(M),\]  
where $\Psi(\delta)$ is independent of $u$ and $\Psi(\delta)\to 0$ as $\delta\to 0$. 
\end{remark}

\subsection{Weak nonnegative scalar curvature }
Under the same assumption about $\Sigma$ as in Theorem \ref{thm1.2}, we can check that $R_g$ is weakly nonnegative. We have the following lemma.
\begin{lemma}\label{lm4.1}
	Let $ M^n $ be a smooth manifold with $g\in C^0\cap W^{1,p}_{loc}(M)$ with $n\le p\le \infty$. Assume $g$ is smooth away from a closed subset $\Sigma$ with $\cH^{n-\frac{p}{p-1}}(\Sigma)<\infty$ if $n\le p<\infty$ or $\cH^{n-1}(\Sigma)=0$ if $p=\infty$, and assume $R_g\ge 0$ on $M\setminus \Sigma$, then $\langle R_g, u\rangle\ge 0$ for any nonnegative, compactly supported $u\in W^{1,p/(p-1)}$. 
\end{lemma}
\begin{proof}
	By definition, we have 
	\begin{align}
		\langle R_g,u\rangle=\int_M \left(-V\cdot \tilde\nabla \left(u \frac{d\mu_g}{d\mu_h}\right)+Fu \frac{d\mu_g}{d\mu_h}\right)d\mu_h.
	\end{align}
	When $n\le p<\infty$, since nonnegative Lipschitz function is dense in $W^{1,p/(p-1)}$, assume $u_i\to u$ in $W^{1,p/(p-1)}$ where $u_i\ge 0$ and Lipschitz. By using Cauchy inequality, it is easy to check that as $i\to \infty$
	\begin{align}
		\Big| \langle R_g,u\rangle-\langle R_g,u_i\rangle\Big|=\Big|\langle R_g, u-u_i\rangle \Big|\to 0.
	\end{align}

	Therefore, to prove the lemma, it suffices to show it holds for nonnegative Lipschitz function $u$. Let us now assume $u\ge 0$ and $|\nabla u|\le L$. Since $u$ has compact support, in particular, we have $|u|\le C(L)$.
	Let $\eta_\epsilon\ge 0$ be a sequence of smooth cutoff functions of $\Sigma$ as in Lemma \ref{l:cut-off} such that 
	\begin{itemize}
		\item[(1)] $\eta_\epsilon\equiv 1$ in a neighborhood of $\Sigma$,
		\item[(2)] ${\rm supp ~}\eta_\epsilon \subset B_{\epsilon}(\Sigma)$ and $0\le \eta_\epsilon\le 1$.
		\item[(3)] $\lim_{\epsilon\to 0}\int_M|\nabla \eta_\epsilon|^{p/(p-1)}d\mu_h=0$. 
	\end{itemize}
	We have 
	\begin{align}
		\langle R_g,u\rangle=\langle R_g,\eta_\epsilon u\rangle+\langle R_g,(1-\eta_\epsilon)u\rangle.
	\end{align}
	Noting that for any $\epsilon>0$ we have 
	\begin{align}
		\langle R_g,(1-\eta_\epsilon)u\rangle=\int_{M\setminus \Sigma} R_g(1-\eta_\epsilon)ud\mu_g\ge   0. 
	\end{align}
	To prove the lemma, it suffices to show that 
	\begin{align}\label{e:limitepsilonRgu0}
		\lim_{\epsilon\to 0}|\langle R_g,\eta_\epsilon u\rangle|=0.
	\end{align}
	Actually, noting that $u$ is bounded and Lipschitz we can estimate 
		\begin{align}
		|\langle R_g,\eta_\epsilon u\rangle|\le& \int_M |V|\cdot \Big|\tilde\nabla (\eta_\epsilon u \frac{d\mu_g}{d\mu_h})\Big|d\mu_h+\int_M |F|\cdot \eta_\epsilon u \frac{d\mu_g}{d\mu_h}d\mu_h\\
		\le &C\int_M |V||\bar \nabla \eta_\epsilon|d\mu_h +C \int_M |V|  |\tilde\nabla g| \eta_\epsilon d\mu_h+C\int_M |V| \eta_\epsilon d\mu_h+C \int_M|F|\eta_\epsilon d\mu_h\\
		\le& C\left(\int_{M\cap B_{\epsilon}(\Sigma)}|\tilde\nabla g|^pd\mu_h\right)^{1/p}\left(\int_M\Big(|\tilde\nabla\eta_\epsilon|^{p/(p-1)}+|\eta_\epsilon|^{p/(p-1)}\Big) d\mu_h\right)^{(p-1)/p}\\
		&+C\left(\int_{M\cap B_{\epsilon}(\Sigma)}|\tilde\nabla g|^pd\mu_h\right)^{2/p}\left(\int_M|\eta_\epsilon|^{p/(p-2)}d\mu_h\right)^{(p-2)/p}.  
	\end{align}
	Letting $\epsilon\to 0$, by the properties of $\eta_\epsilon$ we get \eqref{e:limitepsilonRgu0}. Thus we finish the proof of the lemma for $n\le p<\infty$.\\
	When $p=\infty$, the arguement above still works. (See also  \cite{Lee2015a}.) Thus the lemma follows.
\end{proof}

\section{Positive mass Theorem: Nonnegativity}

Now we use some technical results established above and modify Miao's method \cite{Miao2002} to prove Theorem \ref{thm1.2}. At first, we choose an arbitrary end of $M$. Let us recall a lemma essentially proved in \cite[Lemma 3.2]{Schoen1979}.
\begin{lemma}[\cite{Schoen1979}]\label{lm6.1}
Let $(M,{\tilde g})$ be a complete smooth asymptotically flat manifold, $f,h$ be smooth functions with compact support on $M$, then there exists an $\epsilon>0$, such that if $f$ satisfies
\begin{align*}
\int_M f\xi^2d\mu_{\tilde g}\geq -\epsilon \int_M |\nabla \xi|^2d\mu_{\tilde g}, \forall \xi\in C_0^\infty(M),
\end{align*}
then the equation
\begin{equation}\label{eqqq}
\Delta_{\tilde g}v -fv=h 
\end{equation}
has a solution $v$ satisfying $v=O(r^{2-n})$ as $r\to\infty$. Moreover, we have
\[v=\frac{A}{r^{n-2}}+\omega,\]
where $A$ is a constant, $\omega=O(r^{1-n})$ and $|\partial\omega|=O(r^{-n})$.
\end{lemma}
\begin{remark}
The condition that $f$ and $h$ have compact support can be weakened to a decay condition. However, Lemma \ref{lm6.1} is good enough for our use.
\end{remark}
\begin{remark}
We will let $\tilde g=g_\delta$ when we apply Lemma \ref{lm6.1}. 
\end{remark}
\begin{proof}
Since the proof is almost the same as Schoen-Yau \cite[Lemma 3.2]{Schoen1979}, we only give the proof of the existence of $v$ arguing as Schoen-Yau \cite{Schoen1979}. Suppose that $\Sigma_k$ is the end we take in consideration and there is an asymptotically flat coordinate from $\Sigma_k$ to $\mathbb{R}^n\backslash B_\sigma(0)$. Let us solve the equation
\begin{align}\label{eq6.1}\left\{
\begin{array}{rl}
&\Delta_{\tilde g} v_\sigma-fv_\sigma=h \quad on \quad M^\sigma\\
&v_\sigma=0 \quad on \quad \partial B_\sigma 
\end{array} \right., 
\end{align}
where $M^\sigma=(M\backslash\Sigma_k)\bigcup (\Sigma_k\bigcap B_\sigma)$ and $\sigma>\sigma_0$.
To study the kernel of $\Delta_{\tilde g}-f$, we take $h=0$, and assume $v_\sigma$ is the solution.  Multiplying $v_\sigma$ to the equation and integrating by parts, then we get
\begin{align*}
\int_{M^\sigma} |\nabla v_\sigma|^2d\mu_{\tilde g}&=-\int_{M^\sigma} f v_\sigma^2d\mu_{\tilde g}\\
&\leq \epsilon \int_{M^\sigma} |\nabla v_\sigma|^2d\mu_{\tilde g}
\end{align*}
Therefore, if we take $\epsilon<1$, then we have $|\nabla v_\sigma|\equiv0$ on $M^\sigma$. By the boundary condition, we get that the kernel is trivial. Thus by Fredholm alternative, (\ref{eq6.1}) has a unique solution for general $h\in C_0^\infty(M)$. 

For general $h\in C_0^\infty(M)$, multiplying $v_\sigma$ to (\ref{eq6.1}) and integrating by parts, by Sobolev inequality and Cauchy inequality, we get

\begin{align*}
\int_{M^\sigma} |\nabla v_\sigma|^2d\mu_g&=
-\int_{M^\sigma} f v_\sigma^2d\mu_{\tilde g}-\int_{M^\sigma} h v_\sigma d\mu_{\tilde g}\\
&\leq \epsilon \int_{M^\sigma} |\nabla v_\sigma|^2d\mu_{\tilde g} +\|h\|_{L^{2n/(n+2)}}  \|v_\sigma\|_{L^{2n/(n-2)}}\\
&\leq \epsilon \int_{M^\sigma} |\nabla v_\sigma|^2d\mu_{\tilde g} + C(\tilde g, h) \left(\int_{M^\sigma} |\nabla v_\sigma|^2d\mu_{\tilde g}\right)^\frac{1}{2}\\
&\leq \epsilon \int_{M^\sigma} |\nabla v_\sigma|^2d\mu_{\tilde g} +\epsilon \int_{M^\sigma} |\nabla v_\sigma|^2d\mu_{\tilde g} +C(\tilde g, h)\epsilon^{-1}
\end{align*}
Thus if we choos $\epsilon<\frac{1}{2}$, then we have 
\begin{align*}
\int_{M^\sigma} |\nabla v_\sigma|^2d\mu_{\tilde g}< C(\tilde g, h)
\end{align*}
Thus by Sobolev inequality, we have
\begin{align*}
\|v_\sigma\|_{L^{2n/(n-2)}}< C(\tilde g, h).
\end{align*}
By Moser iteration (see \cite[Theorem 4.1]{han1997elliptic}), we have 
\begin{align*}
\|v_\sigma\|_{C^0}< C(\tilde g, h).
\end{align*}
Thus by Schauder theory (see \cite[Theorem 6.2]{GTru}), we have that $\{v_\sigma|\sigma>\sigma_0\}$ is equicontinuous in $C^2$ topology on any compact subsets of $M$. Thus by Arzela-Ascoli Theorem, there is a $v\in C^2(M)$ and a sequence $\sigma_i\to\infty$ such that $v_{\sigma_i}\to v$ uniformly in $C^2$-norm on any compact subsets of $M$. Thus $v$ solves $\ref{eqqq}$ and therefore $v$ is smooth.

The analysis of the asymptotic behavior of $v$ has no difference with Schoen-Yau \cite[Lemma 3.2]{Schoen1979}. Thus the lemma is proved.
\end{proof}
Now, let $M^n$ ($n\geq3$) be a smooth manifold, and let $g\in C^0\cap W^{1,p}_{loc}(M)$ ($n\le p\le \infty$) be a complete asymptotically flat metric on $M$. Assume $g$ is smooth away from a bounded closed subset $\Sigma$ with $\cH^{n-\frac{p}{p-1}}(\Sigma)<\infty$ if $n\le p<\infty$ or $\cH^{n-1}(\Sigma)=0$ if $p=\infty$, assume that $R_g\ge 0$ on $M\setminus \Sigma$. Let $g_\delta$ be the smooth mollification metric proposed in Lemma \ref{lm2.1}, which converge to $g$ and equal to $g$ outside a compact set $K \supset \Sigma$. Denote $c_n=\frac{n-2}{4(n-1)}$ and $R_{g_\delta}$ to be the scalar curvature of $g_\delta$. Let $\varphi:M\to[0,1]$ be a smooth cut-off function such that $\varphi=1$ on $K$, $\varphi=0$ outside some neighborhood of $K$. 
We consider the equation (see also \cite{LiMa} )
\begin{align}\label{eq6.2}\left\{
\begin{array}{rl}
&\Delta_{g_\delta} u_\delta-c_n\varphi^2 R_{g_\delta} u_\delta=0 \quad on \quad M\\
&\lim_{r\to\infty}u_\delta=1 
\end{array} \right., 
\end{align}

\begin{corollary}\label{cor6.2}
There exists $\delta_0>0$ such that the equation (\ref{eq6.2}) has a positive solution for all $\delta\in(0,\delta_0)$. Moreover, we have
\begin{align*}
u_\delta=1+\frac{A_\delta}{r^{n-2}}+\omega,
\end{align*}
where $A_\delta$ is a constant, $\omega=O(r^{1-n})$ and $|\partial\omega|=O(r^{-n})$.
\end{corollary}
\begin{proof}
Denote $v_\delta=u_\delta-1$, then the equation becomes
\begin{align}\label{eq6.3}
\Delta_{g_\delta} v_\delta-c_n\varphi^2 R_{g_\delta} v_\delta=c_n\varphi^2 R_{g_\delta}
\end{align}
By \cite[Lemma 3.1]{Schoen1979}, there exists a constant $C>0$, such that for any $\xi\in C_0^{\infty}(M)$, the Sobolev inequality
\[\left(\int_M \xi^{\frac{2n}{n-2}}d
\mu_h\right)^{\frac{n-2}{n}}\leq C \int_M|\nabla \xi|^2d
\mu_h\]
holds. Thus by Lemma \ref{lm2.2}, and since $g_\delta$ are uniformly equivalent to $g$, we know that for any $\epsilon>0$, there exists $\delta_0\in(0,\delta_1)$, such that for any $\xi\in C_0^\infty(M)$
\begin{align}
|\langle R_{g_\delta},\xi^2 \rangle-\langle R_g,\xi^2 \rangle|\leq \epsilon\int_M|\nabla \xi |^2d\mu_{g_\delta},\forall \delta\in (0,\delta_0).
\end{align}
Since Lemma \ref{lm4.1} gives  
\[\langle R_g,\xi^2 \rangle\geq 0,\]
we have
\[\langle R_{g_\delta},\xi^2 \rangle\geq -\epsilon\int_M|\nabla \xi |^2d\mu_{g_\delta},\forall \delta\in (0,\delta_0).\]
Thus we can compute that
\begin{align}
\int_M c_n\varphi^2 R_{g_\delta}\xi^2d\mu_{g_\delta}&=c_n\langle R_{g_\delta},\varphi^2\xi^2\rangle\notag\\
&\geq -C\epsilon\int_M|\nabla (\varphi\xi) |^2d\mu_{g_\delta}\notag\\
&=-C\epsilon\left(\int_M|\varphi\nabla \xi |^2d\mu_{g_\delta}+\int_M|\xi\nabla \varphi |^2d\mu_{g_\delta}\right)\notag\\
&\geq -C\epsilon\left(\int_M|\nabla \xi |^2d\mu_{g_\delta}+\int_{{\rm{supp}\varphi}\backslash K}|\xi|^2d\mu_{g_\delta}\right)\notag\\
&\geq -C\epsilon\left(\int_M|\nabla \xi |^2d\mu_{g_\delta}+\left(\int_{{\rm{supp}\varphi}\backslash K}|\xi|^{\frac{2n}{n-2}}d\mu_{g_\delta}\right)^{\frac{n-2}{n}}\right)\label{eq6.5}\\
&\geq -C\epsilon \int_M|\nabla \xi|^2d\mu_{g_\delta}\label{eq6.6}, \forall \delta\in (0,\delta_0),
\end{align}
where $C$ denotes some positive constant which is independent of $\epsilon$ and $\delta$ and varies from line to line, and the last two inequalities follows from H\"older inequality and Sobolev inequality.
Thus by Lemma \ref{lm6.1}, we get the existence and the asymptotical estimate of $u_\delta$. For the positivity of $u_\delta$, we just need to combine the positivity proof in \cite[Lemma 3.3]{Schoen1979} and our proof of Lemma \ref{lm6.1}. See \cite{Schoen1979} for more details.
\end{proof}
\begin{proposition}\label{prop6.3}
Let $u_\delta$ be the positive solution of (\ref{eq6.2}), then
\begin{itemize}
\item[(1)]There exists a $\delta_0>0$, such that for any compact $A\subset M$, there exists a positive constant $C(A)$, such that
\[\int_A u_{\delta}^\frac{2n}{n-2}d\mu_{g_{\delta}}\leq C(A),\forall \delta\in(0,\delta_0).\]
\item[(2)]\[A_\delta=\frac{1}{(2-n)\omega_n}\int_M \big(|\nabla_{g_\delta}u_\delta|^2+c_n\varphi^2 R_{g_\delta} u_\delta^2\big)d\mu_{g_\delta},\]
where $\omega_n$ is the Euclidean volume of the $n-1$ dimensional unit sphere in $\mathbb{R}^n$.
\end{itemize}
\end{proposition}
\begin{proof}
Let us first prove the first statement (1). By the asymptotical behavior of $u_\delta$, we know that $u_\delta$ is bounded. However the $L^\infty$ bound may depend on $\delta$. Let us now show that $u_\delta$ is locally $L^\frac{2n}{n-2}$ where the bound is independent of $\delta$. To prove this, we only need to prove that there exists $\delta_0>0$, such that for any compact $A\subset M$, we have $\int_A u_{\delta}^\frac{2n}{n-2}d\mu_{g_{\delta}}\leq C(A)$, for all $\delta\in(0,\delta_0)$. Assume $v_\delta=u_\delta-1$ as in the proof of Corollary \ref{cor6.2}.
Multiplying $v_\delta$ to both sides of equation (\ref{eq6.3}),  we have
\begin{equation}\label{eq6.7}
\int_M|\nabla v_\delta|^2d\mu_{g_\delta}=-\int_M c_n \varphi^2 R_{g_\delta} v_\delta^2d\mu_{g_\delta}-\int_M c_n\varphi^2R_{g_\delta} v_\delta d\mu_{g_\delta}.
\end{equation}

 By Lemma \ref{lm2.2}, we have that for any $\epsilon>0$, there exists $\delta_0>0$ such that
\begin{align*}
|c_n\langle R_{g_ {\delta}},\varphi^2 v_\delta^2\rangle-c_n\langle R_g,\varphi^2 v_\delta^2\rangle|\leq \epsilon\int_M|\nabla_{g_{\delta}} (\varphi v_ {\delta})|^2d\mu_{g_ {\delta}},\forall \delta\in(0,\delta_0),
\end{align*}
\begin{align*}
|c_n\langle R_{g_ {\delta}},\varphi^2 v_\delta\rangle-c_n\langle R_g,\varphi^2 v_ {\delta}\rangle|\leq \epsilon\left(\int_M|\nabla_{g_{\delta}} (\varphi^2 v_ {\delta})|^\frac{n}{n-1}d\mu_{g_ {\delta}}\right)^\frac{n-1}{n},\forall  \delta\in(0,\delta_0),
\end{align*}
where $c_n\langle R_g,\varphi^2 v_ {\delta}^2\rangle\geq0$, and
\begin{align*}c_n\langle R_g,\varphi^2 v_ {\delta}\rangle&=c_n\langle R_g,\varphi^2 (v_ {\delta}+1)\rangle-c_n\langle R_g,\varphi^2\rangle\\
&\geq0-c_n\int_M \left(-V\cdot \tilde\nabla \left(\varphi^2 \frac{d\mu_g}{d\mu_h}\right)+F\varphi^2 \frac{d\mu_g}{d\mu_h}\right)d\mu_h\\
&\geq -C.
\end{align*}
Thus by equation (\ref{eq6.7}) we have
\begin{align*}\int_M |\nabla_{g_{\delta}} v_ {\delta}|^2d{\mu_{g_ {\delta}}}&\leq \epsilon\int_M|\nabla_{g_{\delta}} (\varphi v_ {\delta})|^2d\mu_{g_ {\delta}}+\epsilon\left(\int_M|\nabla_{g_{\delta}} (\varphi^2 v_ {\delta})|^\frac{n}{n-1}d\mu_{g_ {\delta}}\right)^\frac{n-1}{n}+C\\
&\leq \epsilon\int_M|\nabla_{g_{\delta}} (\varphi v_ {\delta})|^2d\mu_{g_ {\delta}}+C\epsilon \left(\int_{\rm{supp}\varphi}|\nabla_{g_{\delta}}(\varphi^2v_ {\delta})|^2d\mu_{g_ {\delta}}\right)^{\frac{1}{2}}+C\\
&\leq \epsilon\int_M|\nabla_{g_{\delta}} (\varphi v_{\delta})|^2d\mu_{g_ {\delta}}+C\epsilon\left(\int_M |\nabla_{g_{\delta}}(\varphi^2 v_ {\delta})|^2d\mu_{g_ {\delta}}+1\right)+C.
\end{align*}
Since
\begin{align*}
\int_M|\nabla_{g_{\delta}} (\varphi v_ {\delta})|^2d\mu_{g_ {\delta}}&\leq 2\int_M|\nabla_{g_{\delta}}\varphi|^2v_ {\delta}^2d\mu_{g_ {\delta}}+2\int_M\varphi^2|\nabla_{g_{\delta}} v_ {\delta}|^2d\mu_{g_ {\delta}}\\
&\leq 2\| \nabla_{g_{\delta}}\varphi\|_{L^{{n}}}^2\|v_ {\delta}\|_{L^\frac{2n}{n-2}}^2+2\int_M|\nabla_{g_{\delta}} v_ {\delta}|^2d\mu_{g_ {\delta}}\\
&\leq (C\| \nabla_{g_{\delta}}\varphi\|_{L^{n}}^2+2)\int_M|\nabla_{g_{\delta}} v_ {\delta}|^2d\mu_{g_ {\delta}}\\
&\leq C\int_M|\nabla_{g_{\delta}} v_ {\delta}|^2d\mu_{g_ {\delta}},
\end{align*}
where we used that fact that $v_\delta$ is asymptotical to zero and thus the Sobolev inequality  $||v_\delta||_{L^{2n/n-2}(M)}\le C||\nabla v_\delta||_{L^2(M)}$ holds for $v_\delta$.
Similarly, we have
\[\int_M|\nabla_{g_{\delta}} (\varphi^2 v_ {\delta})|^2d\mu_{g_ {\delta}}\leq C\int_M|\nabla_{g_{\delta}} v_ {\delta}|^2d\mu_{g_ {\delta}}.\]
Hence we arrive at 
\begin{equation*}\int_M|\nabla_{g_{\delta}}  v_ {\delta}|^2d\mu_{g_ {\delta}}\leq C,\forall \delta\in(0,\delta_0).
\end{equation*}
Using Sobolev inequality again, we have
\begin{equation*}
\int_M v_ {\delta}^{\frac{2n}{n-2}}d\mu_{g_ {\delta}}\leq C,\forall \delta\in(0,\delta_0).
\end{equation*}
Therefore we have that for any compact $A\subset M$,
\[\int_A u_{\delta}^\frac{2n}{n-2}d\mu_{g_{\delta}}\leq C(A),\forall \delta\in(0,\delta_0).\]

Now we begin the proof of (2).  We multiply $u_\delta$ to both sides of equation (\ref{eq6.2}) and integrate it by parts, then we have
\begin{align*}
-\int_M|\nabla u_\delta|^2d\mu_{g_\delta}+\lim_{r\to\infty}\int_{S_r}u_\delta\frac{\partial u_\delta}{\partial r}d\mu_{g_\delta}-c_n\int_M \varphi^2 R_{g_\delta} u_\delta^2d\mu_{g_\delta}=0.
\end{align*}
Since $u_\delta=1+\frac{A_\delta}{r^{n-2}}+\omega$, where $\omega=O(r^{1-n})$ and $|\partial\omega|=O(r^{-n})$, we have
\[\lim_{r\to\infty}\int_{S_r}u_\delta\frac{\partial u_\delta}{\partial r}d\mu_{g_\delta}=(2-n)\omega_nA_\delta.\]
Thus we get the required result.
\end{proof}
Now, we define the conformal metrics
\[\tilde g_\delta=u_\delta^{\frac{4}{n-2}}g_\delta\]
Then the standard conformal transformation formula shows
\begin{align*}\tilde R_{g_\delta}&=-c_n^{-1}u_\delta^{-\frac{n+2}{n-2}}(\Delta_{g_\delta} u_\delta-c_nR_{g_\delta} u_\delta)\\
&=u_\delta^{1-\frac{n+2}{n-2}}(R_{g_\delta}-\varphi^2 R_{g_\delta})\\
&\geq 0,
\end{align*}
since $\varphi=1$ on $K$ and $R_{g_\delta}=R(g)\geq0$ on $M\backslash K$.
\begin{lemma}\label{lm6.3}
The lower limit of mass of $g_\delta$ is no less than the lower limit of mass of $\tilde g_\delta$.
\end{lemma}
\begin{proof}
By the definition of mass, we can calculate straightforwardly and get the following equation (see \cite[Lemma4.2]{Miao2002}),
\begin{equation}\label{eq6.8}
m(\tilde g_\delta)=m(g_\delta)+(n-1)A_\delta
\end{equation}
By (2) of Proposition \ref{prop6.3}
\[A_\delta=\frac{1}{(2-n)\omega_n}\int_M \big(|\nabla_{g_\delta}u_\delta|^2+c_n\varphi^2 R_{g_\delta} u_\delta^2\big)d\mu_{g_\delta}\]
We calculate as inequality (\ref{eq6.5}), and get 
\begin{equation}\label{eq_1}\int_M c_n\varphi^2 R_{g_\delta} u_\delta^2d\mu_{g_\delta}\geq -C\epsilon\left(\int_M|\nabla u_\delta |^2d\mu_{g_\delta}+\left(\int_{{\rm{supp}}\varphi\backslash K}|u_\delta|^{\frac{2n}{n-2}}d\mu_{g_\delta}\right)^{\frac{n-2}{n}}\right).
\end{equation}
By Proposition \ref{prop6.3} (1), we have that for any $\epsilon>0$, there exists $\delta_0>0$, such that
\[\int_Mc_n\varphi^2R_{g_\delta} u_\delta^2d\mu_{g_\delta}\geq-C\epsilon\int_M|\nabla u_\delta|^2d_{\mu_{g_\delta}}-C\epsilon,\forall\delta\in(0,\delta_0).\]
Therefore we have
\[\uplim_{\delta\to0^{+}}A_\delta\leq0,\]
and thus 
\[\lowlim_{\delta\to0^{+}} m(g_\delta)\geq\lowlim_{\delta\to0^{+}} m(\tilde g_\delta).\]
\end{proof}
Now we can prove the inequality part of Theorem \ref{thm1.2}
\begin{proof}[proof of the inequality part of Theorem \ref{thm1.2}]
Since $g_\delta=g$ on $M\backslash K$, we have
\[m(g_\delta)=m(g)\]
Since $\tilde g_\delta$ has nonnegative scalar curvature, by the classical positive mass theorem, we have 
\[m(\tilde g_\delta)\geq0\]
Thus by Lemma \ref{lm6.3}, we have 
\[m(g)\geq0\]
\end{proof}

\section{Positive mass Theorem: Rigidity}
In this section, we will prove the rigidity part of Theorem \ref{thm1.2} when $m(g)=0$. 
Let us outline the idea of the proof. First, we will follow the idea of Shi-Tam \cite{ShiTam} to show that the manifold is Ricci flat away from $\Sigma$.  Then we will show that the manifold has nonnegative Ricci curvature in RCD sense. Noting that the manifold is asymptotically flat, by volume convergence and volume comparison, we get the rigidity result.  

\subsection{Ricci flat away from the singular set}
Let us  first prove that $m(g)=0$ implies Ricci curvature vanishing away from the singular set.

\begin{lemma}\label{lm6.6}
Assume as Theorem \ref{thm1.2}. If $m(g)=0$, then $\Ric_g\equiv0$ on $M\backslash\Sigma$.
\end{lemma}

\begin{proof}
Suppose that there exists a point $p\in M\backslash \Sigma$ such that $\Ric_g(p)\not=0$, and let $U\subset M\backslash\Sigma$ be a neighborhood of $p$ such that $|\Ric_g|^2\geq \frac{|\Ric_g|^2(p)}{2}$ on $U$. When we construct
the mollification metric, we can take $K$ small enough such that $U\subset M\backslash K$. And when we construct the
cut-off function $\varphi$, we make an additional requirement that $\varphi=1$ on U. We let $\psi\in C_0^{\infty}(U)$ be a cut-off function such that $0\leq\psi\leq1$ and $\psi=1$ on $B_r(p)$, where $r$ is some positive constant such that $B_r(p)\subset U$. Denote $g_{\delta;t}=g_\delta-t\psi \Ric_g$, $g_t=g-t\psi  \Ric_g$, $t\geq 0$, and denote $R_{g_{\delta;t}}$, $R_{g_\delta}$, $R_{g_t}$ as the scalar curvature of $g_{\delta;t}$, $g_\delta$, $g_t$ respectively. Then we have
\[R_{g_{\delta;t}}=R_{g_\delta}-t\Div_{g_{\delta;t}}(\Div_{g_{\delta;t}}(\psi \Ric_g))+t\Delta_{g_{\delta;t}}{\rm{tr}}_{g_{\delta;t}}(\psi \Ric_g)+t\langle \psi \Ric_g,\Ric_{g_{\delta;t}}\rangle_{g_{\delta;t}}+h_\delta,\]
where $|h_\delta|\leq Ct^2$ and ${\rm{supp } } h_\delta\subset U$. Here and below, $C$ and $C_i$ will denote some positive constant depend on $n,g,r$ and are independent of $\delta$, $t$ and $C$ can vary from line to line.

Since $g_\delta=g$ on $U$, we have 
\begin{align}\label{Re}
R_{g_{\delta;t}}=R_{g_\delta}-t\Div_{g_t}(\Div_{g_t}(\psi \Ric_g))+t\Delta_{g_t}{\rm{tr}}_{g_t}(\psi \Ric_g)+t\langle \psi \Ric_g,\Ric_{g_t}\rangle_{g_t}+h,
\end{align}
where $h$ is independent of $\delta$, $|h|\leq Ct^2$ and ${\rm{supp}} h\subset U$.

Let $u_{\delta;t}$ be the solution to equation
\begin{align}\label{eq6.211}\left\{
\begin{array}{rl}
&\Delta_{g_{\delta;t}} u_{\delta;t}-c_n \varphi^2R_{g_{\delta;t}} u_{\delta;t}=0 \quad on \quad M\\
&\lim_{r\to\infty}u_{\delta;t}=1 
\end{array} \right., 
\end{align}
Then the metric $\tilde g_{\delta;t}=u_{\delta;t}^\frac{4}{n-2}g_{\delta;t}$ is $C^2$ on $M$ and has nonnegative scalar curvature. And we have
\begin{equation}
m(\tilde g_{\delta;t})=m(g_{\delta;t})+(n-1)A_{\delta;t},
\end{equation}
where $m(g_{\delta;t})=m(g)=0$, and $A_{\delta;t}=\frac{1}{(2-n)\omega_n}\int_M \left(|\nabla_{g_{\delta;t}} u_{\delta;t} |^2+c_n\varphi^2 R_{g_{\delta;t}}u_{\delta;t}^2\right)d\mu_{g_{\delta;t}}$.
Since $\varphi=1$ on $U$, by (\ref{Re}), we have
\begin{align}\label{ieq6.14}
\int_M\varphi^2 R_{g_{\delta;t}}u_{\delta;t}^2 d\mu_{g_{\delta;t}}=&\int_M \varphi^2 R_{g_\delta} u_{\delta;t}^2 d\mu_{g_{\delta;t}}-t\int_U   u_{\delta;t}^2 \Div_{g_t}(\Div_{g_t}(\psi \Ric_g))d\mu_{g_t}+t\int_U   u_{\delta;t}^2 \Delta_{g_t}{\rm{tr}}_{g_t}(\psi \Ric_g)d\mu_{g_t}\notag\\
&+t\int_U   u_{\delta;t}^2 \langle \psi \Ric_g,\Ric_{g_t}\rangle_{g_t}d\mu_{g_t}+\int_U h u_{\delta;t}^2 d\mu_{g_t}
\end{align}
As the same of Corollary \ref{cor6.2}, by \cite[Lemma 3.1]{Schoen1979}, there exists a constant $C>0$, such that for any $\xi\in C_0^{\infty}(M)$, the Sobolev inequality
\[\left(\int_M \xi^{\frac{2n}{n-2}}d\mu_g\right)^{\frac{n-2}{n}}\leq C \int_M|\nabla \xi|^2d\mu_g\]
holds. Thus by Lemma \ref{lm2.2}, we know that for any $\epsilon>0$, there exists $\delta_0>0$, such that 
\begin{align*}
|\langle R_{g_\delta},\xi^2 \rangle-\langle R_g,\xi^2 \rangle|\leq \epsilon\int_M|\nabla \xi |^2d\mu_{g_\delta},\forall \delta\in (0,\delta_0).
\end{align*}
Since Lemma \ref{lm4.1} gives  
\[\langle R_g,\xi^2 \rangle\geq 0,\]
we have
\[\langle R_{g_\delta},\xi^2 \rangle\geq -\epsilon\int_M|\nabla \xi |^2d\mu_{g_\delta},\forall \delta\in (0,\delta_0).\]
Thus we can fix some $t_0>0$, such that for any $\epsilon>0$, there exists $\delta_0>0$, such that for any $\delta\in(0,\delta_0)$, $t\in(0,t_0)$, we have
\begin{align*}
\int_M  \varphi^2 R_{g_\delta} u_{\delta;t}^2d\mu_{g_{\delta;t}}
&=   \langle R_{g_\delta},\varphi^2 u_{\delta;t}^2\rangle  \notag\\
&\geq -C\epsilon\int_M|\nabla (\varphi u_{\delta;t}) |^2d\mu_{g_\delta}\notag\\
&=-C\epsilon\left(\int_M|\varphi\nabla  u_{\delta;t} |^2d\mu_{g_\delta}+\int_M| u_{\delta;t}\nabla \varphi |^2d\mu_{g_\delta}\right)\notag\\
&\geq -C\epsilon\left(\int_M|\nabla  u_{\delta;t} |^2d\mu_{g_\delta}+\int_{{\rm{supp}\varphi}\backslash K}| u_{\delta;t}|^2d\mu_{g_\delta}\right)\notag\\
&\geq -C\epsilon\left(\int_M|\nabla  u_{\delta;t} |^2d\mu_{g_\delta}+\left(\int_{{\rm{supp}\varphi}\backslash K}| u_{\delta;t}|^{\frac{2n}{n-2}}d\mu_{g_\delta}\right) ^{\frac{n-2}{n}}\right)
\end{align*}

Thus there exists $\delta_0>0$ and $t_0>0$, such that for any $\delta\in(0,\delta_0)$, $t\in(0,t_0)$, we have
\begin{align}\label{eqR2}
\int_M \varphi^2R_{g_{\delta}} u_{\delta;t}^2 d\mu_{g_{\delta;t}}\geq -\frac{1}{10}\int_M|\nabla_{g_{\delta;t}} u_{\delta;t} |^2d\mu_{g_{\delta;t}}-a_\delta \left(\int_{{\rm{supp}\varphi}\backslash K}|u_{\delta;t}|^{\frac{2n}{n-2}}d\mu_{g_{\delta;t}}\right)^{\frac{n-2}{n}},
\end{align}
where $a_\delta$ is a function of $\delta$ which is independent of $t$ and $\lim_{\delta\to0}a_\delta=0$.

Now we will prove that there exists $\delta_0>0$ and $t_0>0$, such that for any compact measurable set $A\subset M$, we have $\int_A u_{\delta;t}^\frac{2n}{n-2}d\mu_{g_{\delta;t}}\leq C(A)$, for all $\delta\in(0,\delta_0)$, $t\in(0,t_0)$.

Let $v_{\delta;t}=u_{\delta;t}-1$, then we have
\begin{equation}\label{eq6.13}
\int_M|\nabla_{g_{\delta;t}} v_ {\delta;t}|^2d\mu_{g_ {\delta;t}}=-\int_M c_n \varphi^2 R_{g_{\delta;t}} v_ {\delta;t}^2d\mu_{g_{\delta;t}}-\int_M c_n\varphi^2R_{g_{\delta;t}} v_{\delta;t} d\mu_{g_ {\delta;t}}.
\end{equation}
We use equation (\ref{eq6.13}) to prove that $\int_M |\nabla_{g_{\delta;t}} v_ {\delta;t}|^2d\mu_{g_ {\delta;t}}$ is uniformly bounded at first. By Lemma \ref{lm2.2}, we have that for any $\epsilon>0$, there exists $\delta_0, t_0>0$ such that
\begin{align*}
|c_n\langle R_{g_ {\delta;t}},\varphi^2 v_ {\delta;t}^2\rangle-c_n\langle R_g,\varphi^2 v_ {\delta;t}^2\rangle|\leq \epsilon\int_M|\nabla_{g_{\delta;t}} (\varphi v_ {\delta;t})|^2d\mu_{g_ {\delta;t}},\forall \delta\in(0,\delta_0), t\in (0, t_0),
\end{align*}
\begin{align*}
|c_n\langle R_{g_ {\delta;t}},\varphi^2 v_ {\delta;t}\rangle-c_n\langle R_g,\varphi^2 v_ {\delta;t}\rangle|\leq \epsilon\left(\int_M|\nabla_{g_{\delta;t}} (\varphi^2 v_ {\delta;t})|^\frac{n}{n-1}d\mu_{g_ {\delta;t}}\right)^\frac{n-1}{n},\forall  \delta\in(0,\delta_0), t\in (0, t_0),
\end{align*}
where $c_n\langle R_g,\varphi^2 v_ {\delta;t}^2\rangle\geq0$, and
\begin{align*}c_n\langle R_g,\varphi^2 v_ {\delta;t}\rangle&=c_n\langle R_g,\varphi^2 (v_ {\delta;t}+1)\rangle-c_n\langle R_g,\varphi^2\rangle\\
&\geq0-c_n\int_M \left(-V\cdot \tilde\nabla \left(\varphi^2 \frac{d\mu_g}{d\mu_h}\right)+F\varphi^2 \frac{d\mu_g}{d\mu_h}\right)d\mu_h\\
&\geq -C.
\end{align*}
Thus by equation (\ref{eq6.13}) we have
\begin{align*}\int_M |\nabla_{g_{\delta;t}} v_ {\delta;t}|^2d{\mu_{g_ {\delta;t}}}&\leq \epsilon\int_M|\nabla_{g_{\delta;t}} (\varphi v_ {\delta;t})|^2d\mu_{g_ {\delta;t}}+\epsilon\left(\int_M|\nabla_{g_{\delta;t}} (\varphi^2 v_ {\delta;t})|^\frac{n}{n-1}d\mu_{g_ {\delta;t}}\right)^\frac{n-1}{n}+C\\
&\leq \epsilon\int_M|\nabla_{g_{\delta;t}} (\varphi v_ {\delta;t})|^2d\mu_{g_ {\delta;t}}+C\epsilon \left(\int_{\rm{supp}\varphi}|\nabla_{g_{\delta;t}}(\varphi^2v_ {\delta;t})|^2d\mu_{g_ {\delta;t}}\right)^{\frac{1}{2}}+C\\
&\leq \epsilon\int_M|\nabla_{g_{\delta;t}} (\varphi v_{\delta;t})|^2d\mu_{g_ {\delta;t}}+C\epsilon\left(\int_M |\nabla_{g_{\delta;t}}(\varphi^2 v_ {\delta;t})|^2d\mu_{g_ {\delta;t}}+1\right)+C.
\end{align*}
Since
\begin{align*}
\int_M|\nabla_{g_{\delta;t}} (\varphi v_ {\delta;t})|^2d\mu_{g_ {\delta;t}}
&\leq 2\int_M|\nabla_{g_{\delta;t}}\varphi|^2v_ {\delta;t}^2d\mu_{g_ {\delta;t}}+2\int_M\varphi^2|\nabla_{g_{\delta;t}} v_ {\delta;t}|^2d\mu_{g_ {\delta;t}}\\
&\leq 2\| \nabla_{g_{\delta;t}}\varphi\|_{L^{{n}}}^2\|v_ {\delta;t}\|_{L^\frac{2n}{n-2}}^2+2\int_M|\nabla_{g_{\delta;t}} v_ {\delta;t}|^2d\mu_{g_ {\delta;t}}\\
&\leq (C\| \nabla_{g_{\delta;t}}\varphi\|_{L^{{n}}}^2+2)\int_M|\nabla_{g_{\delta;t}} v_ {\delta;t}|^2d\mu_{g_ {\delta;t}}\\
&\leq C\int_M|\nabla_{g_{\delta;t}} v_ {\delta;t}|^2d\mu_{g_ {\delta;t}},
\end{align*}
and similarly 
\[\int_M|\nabla_{g_{\delta;t}} (\varphi^2 v_ {\delta;t})|^2d\mu_{g_ {\delta;t}}\leq C\int_M|\nabla_{g_{\delta;t}} v_ {\delta;t}|^2d\mu_{g_ {\delta;t}},\]
we have that there exists $ t_0>0$, such that 
\begin{equation*}\int_M|\nabla_{g_{\delta;t}}  v_ {\delta;t}|^2d\mu_{g_ {\delta;t}}\leq C,\forall \delta\in(0,\delta_0),  t\in(0, t_0).
\end{equation*}
Thus by Sobolev inequality, we have
\begin{equation*}
\int_M v_ {\delta;t}^{\frac{2n}{n-2}}d\mu_{g_ {\delta;t}}\leq C,\forall \delta\in(0,\delta_0), t\in(0,t_0).
\end{equation*}
Therefore we have that for any compact measurable set $A\subset M$,
\begin{equation}\label{uLp}\int_A u_{\delta;t}^\frac{2n}{n-2}d\mu_{g_{\delta;t}}\leq C(A),\forall \delta\in(0,\delta_0), t\in(0,t_0).
\end{equation}
Thus (\ref{eqR2}) becomes
\begin{align}\label{ieq6.21}
\int_M \varphi^2R_{g_{\delta}} u_{\delta;t}^2 d\mu_{g_{\delta;t}}\geq -\frac{1}{10}\int_M|\nabla_{g_{\delta;t}} u_{\delta;t} |^2d\mu_{g_{\delta;t}}-b_\delta ,
\end{align}
where $b_\delta$ is a function of $\delta$ which is independent of $t$ and $\lim_{\delta\to0}b_\delta=0$.

Then, by integration by parts and Cauchy inequality, we have
\begin{align}\label{ieq6.22}
-t\int_U   u_{\delta;t}^2 \Div_{g_t}(\Div_{g_t}(\psi \Ric_g))d\mu_{g_t}&=
t\int_U  \langle\Div_{g_t}(\psi \Ric_g),\nabla_{g_t} u_{\delta;t}^2\rangle d\mu_{g_t}\notag\\
&\geq -t\int_U  \left|\Div_{g_t}(\psi \Ric_g)\right|\cdot|\nabla_{g_t}u_{\delta;t}^2|d\mu_{g_t}\notag\\
&\geq -Ct\int_U 2u_{\delta;t}|\nabla_{g_t} u_{\delta;t}|d\mu_{g_t}\notag\\
&\geq -\frac{1}{10}\int_M|\nabla_{g_t} u_{\delta;t}|^2d\mu_{g_t}-Ct^2 \int_U u_{\delta;t}^2d\mu_{g_t}.
\end{align}

Since $\Delta_{g_t}{\rm{tr}}_{g_t}(\psi \Ric_g)=\Div_{g_t}({\rm{tr}}_{g_t}\nabla_{g_t}(\psi \Ric_g))$, similarly we have
\begin{align}\label{ieq6.222}
t\int_U   u_{\delta;t}^2 \Delta_{g_t}{\rm{tr}}_{g_t}(\psi \Ric_g)d\mu_{g_t} \geq -\frac{1}{10}\int_M|\nabla_{g_t} u_{\delta;t}|^2d\mu_{g_t}-Ct^2 \int_U u_{\delta;t}^2d\mu_{g_t}.
\end{align}
Moreover, we have that there exists $t_0>0$, such that
\begin{align}\label{ieq6.23}
\int_U \langle\psi \Ric_g,\Ric_{g_t}\rangle_{g_t}u_{\delta;t}^2d\mu_{g_t}\geq \frac{|\Ric_g|^2(p)}{4}\int_U \psi  u_{\delta;t}^2d\mu_{g_t}, \forall\delta\in(0,\delta_0),  t\in(0,t_0).
\end{align}
By (\ref{ieq6.14}), (\ref{ieq6.21}), (\ref{ieq6.22}), (\ref{ieq6.222}) and (\ref{ieq6.23}), we have that there exists $t_0>0$, such that
\begin{align*}
\int_M\varphi^2 R_{g_{\delta;t}}u_{\delta;t}^2 d\mu_{g_{\delta;t}}\geq  -\frac{1}{2}\int_U |\nabla_{g_t} u_{\delta;t}|^2d\mu_{g_t}+Ct\int_U \psi u_{\delta;t}^2d\mu_{g_t}-Ct^2\int_U u_{\delta;t}^2d\mu_{g_t}, \forall \delta\in(0,\delta_0), t\in(0,t_0),
\end{align*}
and then we have
\begin{align}\label{ieqA2}
A_{\delta,t}&\leq  \frac{1}{(2-n)\omega_n}\left(\frac{1}{10}\int_M|\nabla_{g_{\delta,t}} u_{\delta;t}|^2d\mu_{g_{\delta,t}}+C_0t\int_U\psi u_{\delta;t}^2d\mu_{g_{t}}-C_1t^2\int_U u_{\delta;t}^2d\mu_{g_{t}}\right)+b_\delta\notag\\ 
&\leq  \frac{1}{(2-n)\omega_n}\left(\frac{1}{10}\int_M|\nabla_{g_{\delta,t}} u_{\delta;t}|^2d\mu_{g_{\delta,t}}+C_0t\int_U\psi u_{\delta;t}^2d\mu_{g_{t}}-C_2t^2\right)+b_\delta\notag \\ 
&=  \frac{1}{(2-n)\omega_n}\left(\frac{1}{10}\int_M|\nabla_{g_{\delta,t}} v_{\delta;t}|^2d\mu_{g_{\delta,t}}+C_0t\int_U\psi u_{\delta;t}^2d\mu_{g_{t}}-C_2t^2\right)+b_\delta\notag \\
&\leq  \frac{1}{(2-n)\omega_n}\left(C_3\left(\int_{B_r(p)}|v_{\delta;t}|^\frac{2n}{n-2}d\mu_{g_{\delta,t}}\right)^\frac{n-2}{n}+C_0t\int_{B_r(p)} u_{\delta;t}^2d\mu_{g_{t}}-C_2t^2\right)+b_\delta\notag \\
&\leq  \frac{1}{(2-n)\omega_n}\left(C_4\fint_{B_r(p)}|v_{\delta;t}|^2d\mu_{g_{\delta,t}}+C_5t\fint_{B_r(p)}u_{\delta;t}^2d\mu_{g_{t}}-C_2t^2\right)+b_\delta
, \forall \delta\in(0,\delta_0), t\in(0,t_0),
\end{align}
where the second inequality follows from (\ref{uLp}) and H\"older inequality.

By H\"older inequality, we have
\begin{align*}
\fint_{B_r(p)}u^2_{\delta;t}d\mu_{g_{\delta;t}}&=\fint_{B_r(p)}(1+v_{\delta;t})^2d\mu_{g_{\delta;t}}\\
&=\fint_{B_r(p)}(1+v_{\delta;t}^2+2v_{\delta;t})d\mu_{g_{\delta;t}}\\
&\ge \fint_{B_r(p)}(1+v_{\delta;t}^2-\frac{1}{2}-2v^2_{\delta;t})d\mu_{g_{\delta;t}}\\
&=\frac{1}{2}-\fint_{B_r(p)}v^2_{\delta;t}d\mu_{g_{\delta;t}}.
\end{align*}

If $\fint_{B_r(p)}v^2_{\delta;t}d\mu_{g_{\delta;t}}\le\frac{1}{10}$, then we have $\fint_{B_r(p)}u^2_{\delta;t}d\mu_{g_{\delta;t}}\ge\frac{4}{10}$, thus by (\ref{ieqA2}), we have
\begin{align*}
A_{\delta;t}\le \frac{1}{(2-n)\omega_n}\left(C_4\fint_{B_r(p)}|v_{\delta;t}|^2d\mu_{g_{\delta,t}}+C_6t-C_2t^2\right)+b_\delta, \forall \delta\in(0,\delta_0), t\in(0,t_0).
\end{align*}
Thus we can choose $t$ small enough such that $C_6t-C_2t^2> 0$, and choose $\delta$ small enough such that $A_{\delta;t}< 0$, then we have $m(\tilde g_{\delta;t})<0$, which is a contradiction.

Otherwise, if $\fint_{B_r(p)}v^2_{\delta;t}d\mu_{g_{\delta;t}}\ge\frac{1}{10}$, then by (\ref{ieqA2}), we have
\begin{align*}
A_{\delta;t}\le \frac{1}{(2-n)\omega_n}\left(C_7-C_2t^2\right)+b_\delta, \forall \delta\in(0,\delta_0), t\in(0,t_0).
\end{align*}
We can still let $t$ and $\delta$ small enough such that $A_{\delta;t}< 0$ and thus $m(\tilde g_{\delta;t})<0$, which makes the contradiction again. Thus the proof of the lemma is completed.
\end{proof}

\subsection{RCD space with nonnegative Ricci curvature}
In this subsection, we will show that our singular space has nonnegative Ricci curvature in RCD sense providing the manifold is Ricci flat away from the singular set. In this subsection, $\Delta_g$ will denote the Dirichlet laplacian taken with respect to metric $g$, and we will denote its domain by $D(\Delta_g)$. We will omit the subscription when there is no ambiguity of the metric.

\begin{theorem}\label{t:RCDnonnegative}
	Let $(M^n,g)$ $(n\ge 3)$ be a smooth manifold with $g\in C^0\cap W^{1,p}_{loc}(M)$ with $n\le p\le\infty$. Assume $g$ is smooth and Ricci flat away from a closed subset $\Sigma$ with $\cH^{n-\frac{p}{p-1}}(\Sigma)<\infty$ when $n\le p<\infty$ 
	and $\cH^{n-1}(\Sigma)=0$ when $p=\infty$, and assume $g$ is asymptotically flat. Then $(M^n,g)$ as a metric measure space with Lebesgue measure has nonnegative Ricci curvature in the sense of RCD. 
\end{theorem}

Let us recall the definition of RCD space, see (\cite{RCD,AMS,EKS15}). In general, this is defined on metric measure space. In our setting, we will consider manifold with $C^0$-metric which has a natural metric and measure structure.

\begin{definition}[RCD Ricci lower bound]\label{R}Let $K$ be some real constant. For a Riemannian manifold $(M^n,g)$ with volume measure and $g\in C^0(M)$, we say that it is a RCD($K,n$) space, or say that it has Ricci curvature not less than $K$ in the sense of RCD, if 
\begin{itemize}
		\item[(1)] it is infinitesimally Hilbertian,
		\item[(2)] for some $C>0$, and some point $p\in M$, it holds $\mu_g(B_r(x))\le e^{Cr^2}$, for any $r>0$, where $\mu_g$ is Lebesgue measure taken with respect to $g$,
		\item[(3)] for any $f\in W^{1,2}(M)$ satisfying $|\nabla f|\in L^\infty(M)$, it admits a Lipschitz representative $\tilde f$ with $\Lip (\tilde f)\le \|\nabla f\|_{L^\infty(M)}$,
		\item[(4)] for any $f\in D(\Delta)$ with $\Delta f\in W^{1,2}(M)$, and for any $\varphi \in L^\infty(M)\cap D(\Delta)$ with $\varphi \ge 0$, $\Delta \varphi\in L^\infty(M)$, the Bochner inequality
		\[\frac{1}{2}\int_M |\nabla f|^2\Delta \varphi d\mu_g\ge \frac{1}{n}\int_M (\Delta f)^2\varphi d\mu_g+\int_M \varphi  \left( \langle \nabla f,\nabla \Delta f\rangle +K|\nabla f|^2 \right)d\mu_g\]
		holds.
	\end{itemize}
\end{definition}

\begin{remark}
For a Riemannian manifold $(M^n,g)$ with volume measure and $g\in C^0(M)$, (1) and (3) in Definition \ref{R} hold automatically. If $g$ is asymptotically flat, then (2) holds. Therefore, to prove Theorem \ref{t:RCDnonnegative}, it only needs to check a weak Bochner inequality (see Proposition \ref{p:weakbochnerinequaliyt}, see also \cite{BKMR} ). To approach this, we need some apriori estimates. First we requires gradient estimates for harmonic functions and Bochner inequality for harmonic functions. 
\end{remark}

\begin{proposition}\label{p:gradientdestimate_Deltabound}
	Assume as Theorem \ref{t:RCDnonnegative} and $u\in D(\Delta) $ satisfying $\Delta u\in L^\infty(B_1(x))$. Then $|\nabla u|$ is bounded. Moreover, if $\Delta u=0$, then we have  
	\begin{align}
		\sup_{B_{1/2}(x)}|\nabla u|^2\le C\fint_{B_1(x)}|\nabla u|^2dx,
	\end{align}
	where $C=C(n,g)$ depends only on $n$, $\|g\|_{W^{1,p}}$ and the $L^2$-Sobolev constant in $B_1(x)$.
\end{proposition}
\begin{remark}
	We should point out that $p=n$ and $p=\infty$ are two critical cases. In the following lemmas, sometimes we will deal with these two cases separately.  When $n<p\le \infty$, one can use approximation to show $|\nabla u|$ is bounded(see Lemma \ref{l:gradient_plargern}). However, when $p=n$, we need to use the fact that $g$ is Ricci flat away from $ \Sigma$. 
\end{remark}
\begin{lemma}\label{l:laplaciancovergence}
	Let $(M^n,g_i)$ satisfy $g_i\in C^0$ and $g_i\to g$ in $C^0$-sense. Suppose that $u_i\in D(\Delta_{g_i})$ and $u\in D(\Delta_g)$. Let $\Delta_{g_i}u_i=f_i$ on $B_1(x)$ with $f_i\to f\in L^2$ and $u_i\to u$ in $W^{1,2}$, then $\Delta_gu=f$. 
\end{lemma}
\begin{proof}
The result follows directly. Actually, let $\varphi\in C_0^1(B_1(x))$ be a test function. Then 
\begin{align}
	\int_{B_1(x)}\langle \nabla  u_i,\nabla  \varphi\rangle_{g_i}d\mu_{g_i}=-\int_{B_1(x)}f_i \varphi d\mu_{g_i}.
\end{align}
Since $g_i\to g$ in $C^0$ and $u_i\to u$ in $W^{1,2}$-sense, letting $i\to \infty$, we get 
\begin{align}
	\int_{B_1(x)}\langle \nabla  u,\nabla  \varphi\rangle_{g}d\mu_{g}=-\int_{B_1(x)}f \varphi d\mu_{g}.
\end{align}
This means $\Delta_g u=f$ which proves the lemma. 
\end{proof}

For $n<p\le \infty$, we can use approximation argument to get the following gradient estimates. 
\begin{lemma}\label{l:gradient_plargern}
	Let $(M^n,g)$ be manifold with $g\in C^0\cap W^{1,p}_{loc}(M)$ with $n< p\le\infty$. Assume $u\in D(\Delta) $ and $\Delta u\in L^\infty(B_1(x))$, then $|\nabla u|$ is locally bounded. Moreover, if $\Delta u=0$, then we have  
	\begin{align}
		\sup_{B_{1/2}(x)}|\nabla u|^2\le C\fint_{B_1(x)}|\nabla u|^2dx,
	\end{align}
	where $C=C(n,g)$ depends only on the $L^2$-Sobolev constant in $B_1(x)$.
\end{lemma}
\begin{proof}
	Since the result is local, we will only consider the estimate around $x$. Let $B_r(x)\subset B_{1/2}(x)$ be in a local coordinate. Let $f_i\in C^\infty(M)$ be a sequence of smooth functions converging in $L^2$ sense to $f:=\Delta u$ and satisfying $\sup_{B_1(x)}|f_i|\le 2\sup_{B_1(x)}|\Delta u|+1$. Let $g_i$ be a sequence of smooth metrics converging in $W^{1,p}$-sense to $g$. Let us solve $\Delta_{g_i}u_i=f_i$ in $B_r(x)$ with $u_i=u$ on $\partial B_r(x)$. We will first show $u_i\to u$ in $W^{1,2}(B_r(x))$-sense and then show that $u_i\to u$ pointwisely in $B_{r/2}(x)$ and $|\nabla u_i|$ has uniform bounds on $B_{r/2}(x)$. This gives the bound of $|\nabla u|$.

	 Since $u\in W^{1,p}$ with $p>n$, by Sobolev embedding, $u$ is bounded. Applying maximum principle to each $\Delta_{g_i}u_i=f_i$, we get $|u_i|$ is bounded by the bound of $u$ and the bound of $f_i$.  
Multiplying $u-u_i$ to $\Delta_gu=f$ and $\Delta_{g_i}u_i=f_i$, and integrating by parts, we get 
	\begin{align}
		\int_{B_r(x)}\langle \nabla(u-u_i),\nabla u\rangle_gd\mu_g=-\int_{B_r(x)}f(u-u_i)d\mu_g\\
		\int_{B_r(x)}\langle \nabla(u-u_i),\nabla u_i\rangle_{g_i}d\mu_{g_i}=-\int_{B_r(x)}f_i(u-u_i)d\mu_{g_i}.
	\end{align}
	Hence 
	\begin{align}
		\int_{B_r(x)}\langle \nabla(u-u_i),\nabla (u-u_i)\rangle_{g_i}d\mu_{g_i}=&\int_{B_r(x)}\langle \nabla(u-u_i),\nabla u\rangle_{g_i}d\mu_{g_i}-\int_{B_r(x)}\langle \nabla(u-u_i),\nabla u_i\rangle_{g_i}d\mu_{g_i}\\
		=&\int_{B_r(x)}\langle \nabla(u-u_i),\nabla u\rangle_{g_i}d\mu_{g_i}+\int_{B_r(x)}f_i(u-u_i)d\mu_{g_i}\\
		=&\int_{B_r(x)}\langle \nabla(u-u_i),\nabla u\rangle_{g_i}d\mu_{g_i}-\int_{B_r(x)}\langle \nabla(u-u_i),\nabla u\rangle_gd\mu_g\\
		&+\int_{B_r(x)}f_i(u-u_i)d\mu_{g_i}-\int_{B_r(x)}f(u-u_i)d\mu_g\\
		=&\int_{B_r(x)}\partial_{\alpha}(u-u_i)\partial_{\beta}u\left(g^{\alpha\beta}_i\sqrt{\det(g_{i,\alpha\beta})}-g^{\alpha\beta}\sqrt{\det(g_{\alpha\beta})}\right)dy\\
		&+\int_{B_r(x)}(u-u_i)\left(f_i\sqrt{\det(g_{i,\alpha\beta})}-f\sqrt{\det(g_{\alpha\beta})}\right)dy.
	\end{align}
	This implies that $|\nabla u_i|$ has uniform bounded $L^2$-norm. Moreover, letting $i\to \infty$, noting that $g_i\to g$ in $C^0$ sense we also get that 
	\begin{align}
		\lim_{i\to \infty}\int_{B_r(x)}\langle \nabla(u-u_i),\nabla (u-u_i)\rangle_{g_i}d\mu_{g_i}=0.
	\end{align}
	Since $u_i-u=0$ on $\partial B_r(x)$, by Poincare inequality, we get $u_i\to u$ in $W^{1,2}$-sense.  Let us now show uniformly interior estimates for $u_i$. Since $f_i$ is smooth and $g_i$ is smooth, the function $u_i$ is smooth in $B_r(x)$.  By Bochner formula, we have 
	\begin{align}
		\frac{1}{2}\Delta_{g_i} |\nabla u_i|^2=|\nabla^2u_i|^2+\langle \nabla \Delta_{g_i} u_i,\nabla u_i\rangle_{g_i}+\Ric_{g_i}(\nabla u_i,\nabla u_i).
	\end{align}
	Since $g_i$ has uniform $W^{1,p}$-norm on $B_r(x)$, by elliptic estimates we can get that (see the details in Lemma \ref{l:gradient_W1pmetric})
	\begin{align}\label{ieq4.24}
		\sup_{B_{r/2}(x)}|\nabla u_i|^2\le C(n,g_i)\left(\fint_{B_r(x)}|\nabla u_i|^2d\mu_{g_i}+r^2\sup_{B_r(x)}|f_i|^2\right).
	\end{align}
	where the constant $C(n,g_i)$ depends only on the $W^{1,p}$-norm of $g_i$ on $B_r(x)$, which is uniform. Thus we get uniform estimate for $|\nabla u_i|$ on $B_{r/2}(x)$. By Arzela-Ascoli theorem, up to a subsequence $u_i$ converges to a Lipschitz function $\tilde{u}$ pointwisely on $B_{r/2}(x)$. However, $u_i\to u$ in $W^{1,2}$-sense on $B_r(x)$. Thus $u=\tilde{u}$ and hence $u$ is Lipschitz which finishes the proof.

	Moreover, if $\Delta u=0$, then we can let $f_i=0$ on $M$ for each $i$, and (\ref{ieq4.24}) gives
	\begin{align} 
		\sup_{B_{1/2}(x)}|\nabla u_i|^2\le C(n,g_i)\left(\fint_{B_1(x)}|\nabla u_i|^2d\mu_{g_i} \right).
	\end{align}

Choose an arbitrary point $y\in B_{1/2}(x)$, let $r_y$ be a positive small enough such that $B_{2r_y}(y)\subset B_{1/2}(x)$ and $\sup_{B_{2r_y}(y)}|\nabla u_i|^2\le C(n,g)\fint_{B_1(x)}|\nabla u|^2d\mu_g$ implies ${\rm{Lip}}_{B_{r_y}(y)}u_i\le C(n,g_i)\left(\fint_{B_1(x)}|\nabla u|^2d\mu_{g_i}\right)^\frac{1}{2}$, where ${\rm{Lip}}_{B_{r_y}(y)}u_i$ means the Lipschitz constant of $u_i$ on $B_{r_y}(y)$. Then we have ${\rm{Lip}}_{B_{r_y}(y)}u \le C(n,g)\left(\fint_{B_1(x)}|\nabla u|^2d\mu_g\right)^\frac{1}{2}$, where ${\rm{Lip}}_{B_{r_y}(y)}u $ means the Lipschitz constant of $u $ on $B_{r_y}(y)$, and the constant $C(n,g)$ depends only on the $W^{1,p}$-norm of $g $ on $B_1(x)$. Since $|\nabla u|(y)\le {\rm{Lip}}_{B_{r_y}(y)}u $, we have
\begin{align}
		|\nabla u|^2(y)\le C(n,g)\fint_{B_1(x)}|\nabla u|^2d\mu_g.
	\end{align}
	Since $y$ is arbitrary in $y\in B_{1/2}(x)$, we have
\begin{align}
		\sup_{B_{1/2}(x)}|\nabla u |^2\le C(n,g)\fint_{B_1(x)}|\nabla u|^2d\mu_g,
	\end{align}
	which complete the proof of the Lemma
\end{proof}

Since $g\in C^0$, we can use elliptic $L^p$-theory to get some apriori estimates for $u$.
\begin{lemma}\label{l:aprioriW2q_harmonic}
	Assume as Proposition \ref{p:gradientdestimate_Deltabound}. Then 
$u\in W^{2,q}$ for any $q<p$ when $p=n$ or $p=\infty$, and $u\in W^{2,p}$ for $n<p<\infty$.
\end{lemma}
\begin{remark}
	In the application below, we only require a priori that $u\in W^{2,q}$ for some $q>2$ and $|\nabla u|\in L^q$ for any $q<\infty$. Since $n\ge 3$, this alway holds by Lemma \ref{l:aprioriW2q_harmonic}.
\end{remark}

\begin{proof}Since the result is local, we will only prove the result around $x$.  Let us assume $B_r(x)$ is contained in a coordinate neighborhood.  Let us begin the proof basing on the above observations.  In a local coordinate we know that $u$ satisfies in distribution sense that
	\begin{align}
		\frac{1}{\sqrt{\det(g_{ij})}}\partial_i \left(g^{ij}\sqrt{\det(g_{ij})}\partial_ju\right)=\Delta u.
	\end{align}
	Since $g^{ij}\in C^0$, by $W^{1,q}$-estimate for divergence form elliptic equation \cite{CaPe}, we have $u\in W^{1,q}$ for any $1\le q<\infty$. 
	On the other hand, since $(M,g)$ is smooth away from $\Sigma$, by standard elliptic estimates we get that $u\in W^{2,q}_{loc}$ on $B_r(x)\setminus \Sigma$. Let $v=\varphi u$ where $\varphi$ is a smooth cut-off function with support in $B_r(x)$. Noting  on $B_r(x)\setminus \Sigma$ that
	\begin{align}
		g^{ij}\partial_i\partial_j u-\frac{1}{\sqrt{\det(g_{ij})}}\partial_i \left(g^{ij}\sqrt{\det(g_{ij})}\right)\partial_ju=\Delta u,
	\end{align}
	we get on $B_r(x)\setminus \Sigma$ that 
	\begin{align}\label{e:vQ}
		g^{ij}\partial_i\partial_j v=2g^{ij}\partial_i\varphi \partial_ju+ug^{ij}\partial_i\partial_j\varphi+\frac{\varphi}{\sqrt{\det(g_{ij})}} \partial_i \left(g^{ij}\sqrt{\det(g_{ij})}\right)\partial_ju+\Delta(\varphi u):=Q(u,\partial u,\partial g,\varphi),
	\end{align}
	Then $Q\in L^{q}$ for any  $q<p$.
By solving $g^{ij}\partial_i\partial_jw=Q(u,\partial u,\partial g,\varphi)$ on $B_r(x)$ with $w=v=0$ on $\partial B_r(x)$, we have by Theorem 9.15 of \cite{GTru} that $w\in W^{2,q}$ for any $q<p$. Furthermore, let us check as Shi-Tam \cite{ShiTam} that $w-v\equiv 0$, which will imply $v\in W^{2,q}$ for any $q<p$. Actually, 
 we have in $B_r(x)\setminus \Sigma$ that
\begin{align}\label{e:diveregencew-v}
	g^{ij}\partial_i\partial_j(w-v)=0,
\end{align}
For any $\epsilon>0$,  assume $\psi_\epsilon$ is a cut-off function of $\Sigma$ from Lemma \ref{l:cut-off} satisfying $\psi_\epsilon\equiv 1$ on $M\setminus B_\epsilon(\Sigma)$ and $\psi_\epsilon$ vanishes in a neighborhood of $\Sigma$, and $\lim_{\epsilon\to 0}\int_M|\nabla \psi_\epsilon|^q(x)dx=0$ with $q=\frac{p}{p-1}$. 
Multiplying $(w-v)\psi_\epsilon$ to both sides of \eqref{e:diveregencew-v} and integrating by parts, we get 
\begin{align}\label{e:partialw-v}
	-\int_{B_r(x)} g^{ij}\psi_\epsilon\partial_i(w-v)\partial_j(w-v)d\mu_g= \int_{B_r(x)}g^{ij}(w-v)\partial_i\psi_{\epsilon}\partial_j(w-v)d\mu_g+\int_{B_r(x)}(w-v)\psi_\epsilon \partial_ig^{ij}\partial_j(w-v)d\mu_g
\end{align}
When $n\le p<\infty$, noting that $w-v\in W^{1,q}$ for any $q<\infty$ and $\lim_{\epsilon\to 0}\int_M|\nabla \psi_\epsilon|^{q'}(x)dx=0$ for some $q'>1$, letting $\epsilon\to 0$, we have
\begin{align}
	\lim_{\epsilon\to 0}|\int_{B_r(x)}g^{ij}(w-v)\partial_i\psi_{\epsilon}\partial_j(w-v)d\mu_g|\le \lim_{\epsilon\to 0}C\left(\int_{B_r(x)}|\nabla\psi_{\epsilon}|^{q'}|d\mu_g\right)^{1/q'}\left(\int_{B_r(x)}|\nabla (w-v)|^{q'/(q'-1)}d\mu_g\right)^{1-1/q'}=0.
\end{align}
When $p=\infty$, by Lemma \ref{l:gradient_plargern}, $v-w\in W^{1,\infty}$. Noting that $\lim_{\epsilon\to 0}\int_M|\nabla \psi_\epsilon|(x)d\mu_g=0$, letting $\epsilon\to 0$, we have 
\begin{align}
	\lim_{\epsilon\to 0}|\int_{B_r(x)}g^{ij}(w-v)\partial_i\psi_{\epsilon}\partial_j(w-v)dd\mu_g|\le \lim_{\epsilon\to 0}C\int_{B_r(x)}|\nabla\psi_{\epsilon}|d\mu_g=0.
\end{align}
Therefore, we get from \eqref{e:partialw-v} that
\begin{align}
		-\int_{B_r(x)} g^{ij}\partial_i(w-v)\partial_j(w-v)d\mu_g=\int_{B_r(x)}(w-v) \partial_ig^{ij}\partial_j(w-v)d\mu_g.
\end{align}
Therefore, by Cauchy inequality we get 
\begin{align}
	\int_{B_r(x)} g^{ij}\partial_i(w-v)\partial_j(w-v)d\mu_g\le \left(\int_{B_r(x)}(w-v)^{2n/(n-2)}d\mu_g\right)^{(n-2)/2n}\left(\int_{B_r(x)}|\partial_j(w-v)|^2d\mu_g\right)^{1/2} \left(\int_{B_r(x)}|\partial g^{ij}|^nd\mu_g\right)^{1/n}.
\end{align}
Since $g_{ij}\in W^{1,n}$, for any $\delta>0$, by choosing $r=r(\delta)$ small we have $\int_{B_r(x)}|\partial g^{ij}|^ndx\le \delta^n$. Therefore we get 
\begin{align}
	\int_{B_r(x)} |\partial_i(w-v)|^2d\mu_g\le C\delta \left(\int_{B_r(x)}(w-v)^{2n/(n-2)}d\mu_g\right)^{(n-2)/2n}\left(\int_{B_r(x)}|\partial_j(w-v)|^2d\mu_g\right)^{1/2}.
\end{align}
Thus 
\begin{align}
	\int_{B_r(x)} |\partial_i(w-v)|^2d\mu_g \le C\delta^2\left(\int_{B_r(x)}(w-v)^{2n/(n-2)}d\mu_g\right)^{(n-2)/n}\le C\delta^2 \int_{B_r(x)} |\partial_i(w-v)|^2d\mu_g,
\end{align}
where we have used Sobolev inequality $\left(\int_{B_r(x)}f^{2n/n-2}d\mu_g\right)^{(n-2)/n}\le C_g \int_{B_r(x)}|\nabla f|^2d\mu_g$ for compact support function $f=w-v$.  Noting that the constant $C$ is independent of $r$, if $\delta$ is small enough, $\int_{B_r(x)} |\partial_i(w-v)|^2d\mu_g$ must vanish.  Since $w-v\equiv 0$ on $\partial B_r(x)$, hence we have proved $w-v\equiv 0$.  Thus $v\in W^{2,q}$ for any $q<p$. In particular, $u\in W^{2,q}$ for any $q<p$.

Furthermore, for $n<p<\infty$, noting that $u\in W^{2,q}$ for $q<p$, we can choose $q>n$, thus by Sobolev embedding, $|\nabla u|$ is bounded (see also  Lemma \ref{l:gradient_plargern} ). Therefore, the function $Q$ in \eqref{e:vQ} satisfies $Q\in L^p$. The same argument as above gives that $u\in W^{2,p}$. This completes the proof of the lemma.
\end{proof}

\begin{remark}
	Note that in a local coordinate $\nabla^2u(\partial_i,\partial_j)=\partial_i\partial_j u-\Gamma_{ij}^k\partial_ku$.  If $u\in W^{2,p}$ and $g\in W^{1,p}$ then $|\nabla^2u|\in L^p$.
\end{remark}
\begin{remark}[$W^{2,q}$-estimates]\label{r:W2qestimate}
	By the same argument as above, we see that if $\Delta u\in L^\infty$ and $|\nabla u|\in L^\infty$ then $u\in W^{2,p}$ when $n\le p<\infty$ and $u\in W^{2,q}$ for any $q<\infty$ when $p=\infty$. 
\end{remark}

\begin{lemma}\label{l:weakbochner_boundedlap}
	Assume as Theorem \ref{t:RCDnonnegative}, $n\le p < \infty$ and $u\in D(\Delta) $ satisfying $\Delta u\in L^\infty$. Then for each $s>0$, the following distributional Bochner inequality holds:
	\begin{align}\label{e:bochnerinequalityharmonicbounded}
\int_{B_1(x)} \langle \nabla \varphi,\nabla \sqrt{|\nabla u|^2+s}\rangle d\mu_g\le & \int_{B_1(x)}\left(\frac{\varphi(\Delta u)^2}{\sqrt{|\nabla u|^2+s}}-\varphi\frac{\Delta u}{|\nabla u|^2+s} \langle \nabla u,\nabla \sqrt{|\nabla u|^2+s}\rangle +\frac{\Delta u}{\sqrt{|\nabla u|^2+s}} \langle  \nabla \varphi,\nabla u\rangle\right)d\mu_g,
	\end{align}
	where $\varphi\in W^{1,2}(B_1(x))$ is any nonnegative compactly support function.
\end{lemma}
\begin{proof}
	By Lemma \ref{l:aprioriW2q_harmonic}, $u\in W^{2,q}$ for any $q<p$. Let $u_i\in C^\infty$ be a sequence smooth functions converging in $W^{2,q}$-sense to $u$.  Let us show for each $u_i$ that \eqref{e:bochnerinequalityharmonicbounded}
 holds. Actually, since $u_i\in C^\infty$, then $|\nabla^2u_i|\in L^p$ and $|\nabla u_i|\in L^\infty$ for each $n\le p<\infty$. By direct computations, we can get the Bochner formula  $B_1(x)\setminus \Sigma$ that 
	\begin{align}
		\Delta |\nabla u_i|^2=2|\nabla^2u_i|+2\langle \nabla \Delta u_i,\nabla u_i\rangle.
		\end{align}
		In particular, for any $s>0$ we have weakly on $B_1(x)\setminus \Sigma$ that 
	\begin{align}\label{e:weakDeltasqrtnablau}
		\Delta \sqrt{|\nabla u_i|^2+s}\ge \frac{\langle \nabla \Delta u_i,\nabla u_i\rangle}{\sqrt{|\nabla u_i|^2+s}}.
	\end{align}
	Let us show that this holds weakly on $B_1(x)$. For any $\epsilon>0$,  assume $\psi_\epsilon$ is a cut-off function of $\Sigma$ from Lemma \ref{l:cut-off} satisfying $\psi_\epsilon\equiv 1$ on $M\setminus B_\epsilon(\Sigma)$ and $\psi_\epsilon$ vanishes in a neighborhood of $\Sigma$, and $\lim_{\epsilon\to 0}\int_M|\nabla \psi_\epsilon|^q(x)dx=0$ with $q=\frac{p}{p-1}$. For any nonnegative $\varphi\in C^1_0(B_1(x))$, multiplying $\varphi \psi_\epsilon$ to both sides of \eqref{e:weakDeltasqrtnablau}
 and integrating by parts, we get 
 \begin{align}
 	\int_{B_1(x)} \langle \nabla (\varphi\psi_\epsilon),\nabla \sqrt{|\nabla u_i|^2+s}\rangle d\mu_g \le & \int_{B_1(x)}\left(\frac{(\varphi\psi_\epsilon)(\Delta u_i)^2}{\sqrt{|\nabla u_i|^2+s}}-(\varphi\psi_\epsilon)\frac{\Delta u_i}{|\nabla u_i|^2+s} \langle \nabla u_i,\nabla \sqrt{|\nabla u_i|^2+s}\rangle +\frac{\psi_\epsilon\Delta u_i}{\sqrt{|\nabla u_i|^2+s}} \langle  \nabla \varphi,\nabla u_i\rangle\right)d\mu_g\\  	&+\int_{B_1(x)}\frac{\varphi\Delta u_i}{\sqrt{|\nabla u_i|^2+s}} \langle  \nabla \psi_\epsilon,\nabla u_i\rangle d\mu_g.
 \end{align}

Since $|\nabla^2u_i|+|\Delta u_i|\in L^p$ and $|\nabla u_i|\in L^\infty$, letting $\epsilon\to 0$, we get 
 \begin{align}\label{e:distribution_nablau+1}
 	\int_{B_1(x)} \langle \nabla \varphi,\nabla \sqrt{|\nabla u_i|^2+s}\rangle d\mu_g \le & \int_{B_1(x)}\left(\frac{\varphi(\Delta u_i)^2}{\sqrt{|\nabla u_i|^2+s}}-\varphi\frac{\Delta u_i}{|\nabla u_i|^2+s} \langle \nabla u_i,\nabla \sqrt{|\nabla u_i|^2+s}\rangle +\frac{\Delta u_i}{\sqrt{|\nabla u_i|^2+s}} \langle  \nabla \varphi,\nabla u_i\rangle\right)d\mu_g.
 \end{align}

Noting that $u_i\to u$ in $W^{2,q}$-sense for any $q<p$. Moreover, since $\varphi\in C^1_0(B_1(x))$, letting $i\to \infty$, we get 
 \begin{align}\label{e:distribution_nablauc1varphi}
 	\int_{B_1(x)} \langle \nabla \varphi,\nabla \sqrt{|\nabla u|^2+s}\rangle d\mu_g\le & \int_{B_1(x)}\left(\frac{\varphi(\Delta u)^2}{\sqrt{|\nabla u|^2+s}}-\varphi\frac{\Delta u}{|\nabla u|^2+s} \langle \nabla u,\nabla \sqrt{|\nabla u|^2+s}\rangle +\frac{\Delta u}{\sqrt{|\nabla u|^2+s}} \langle  \nabla \varphi,\nabla u\rangle\right)d\mu_g.
 \end{align}

Let $\varphi\in W^{1,2}$ be nonnegative compactly support function. Assume $\varphi_\alpha\in C_0^1(B_1(x))$ is a sequence of nonnegative functions approximating $\varphi$ in $W^{1,2}$-sense. Noting that $\Delta u\in L^\infty$, applying each $\varphi_\alpha$ to \eqref{e:distribution_nablauc1varphi} and letting $\alpha\to \infty$, we conclude that 
 \begin{align} 	\int_{B_1(x)} \langle \nabla \varphi,\nabla \sqrt{|\nabla u|^2+s}\rangle d\mu_g\le & \int_{B_1(x)}\left(\frac{\varphi(\Delta u)^2}{\sqrt{|\nabla u|^2+s}}-\varphi\frac{\Delta u}{|\nabla u|^2+s} \langle \nabla u,\nabla \sqrt{|\nabla u|^2+s}\rangle +\frac{\Delta u}{\sqrt{|\nabla u|^2+s}} \langle  \nabla \varphi,\nabla u\rangle\right)d\mu_g,
 \end{align}
 holds also for $W^{1,2}$ functions. This completes the proof.
\end{proof}

Now we are ready to prove Proposition \ref{p:gradientdestimate_Deltabound}. 
\begin{proof}[Proof of Proposition \ref{p:gradientdestimate_Deltabound} ]

If $n<p\le \infty$, Proposition \ref{p:gradientdestimate_Deltabound} follows immediately from Lemma \ref{l:gradient_plargern}.

If $p=n$, by the Bochner inequality in Lemma \ref{l:weakbochner_boundedlap}, the Proposition is standard by Moser iteration. For the sake of convenience, we give a proof here.  By Lemma \ref{l:weakbochner_boundedlap}, we have for any $s>0$ and nonnegative $\varphi\in W^{1,2}_0(B_1(x))$ that
\begin{align*}
 	\int_{B_1(x)} \langle \nabla \varphi,\nabla \sqrt{|\nabla u|^2+s}\rangle d\mu_g  \le & \int_{B_1(x)}\left(\frac{\varphi(\Delta u)^2}{\sqrt{|\nabla u|^2+s}}-\varphi\frac{\Delta u}{|\nabla u|^2+s} \langle \nabla u,\nabla \sqrt{|\nabla u|^2+s}\rangle +\frac{\Delta u}{\sqrt{|\nabla u|^2+s}} \langle  \nabla \varphi,\nabla u\rangle\right)d\mu_g .
 \end{align*}
 We denote $v= \sqrt{|\nabla u|^2+s}$, $f=\frac{(\Delta u)^2}{\sqrt{|\nabla u|^2+s}}$.
 Then (5.18) becomes
 \begin{align}
\int_{B_1(x)} \langle \nabla \varphi,\nabla v\rangle d\mu_g  \le \int_{B_1(x)}\left(\varphi f-\varphi \frac{\Delta u}{|\nabla u|^2+s} \langle \nabla u,\nabla v\rangle + v^{-1}\Delta u\langle  \nabla \varphi, \nabla u\rangle\right)d\mu_g .
 \end{align}
 If $\Delta u\equiv 0$, we will let $s\to 0^{+}$ in the end of the following argument. If $\Delta u\not\equiv 0$, we denote $D=\sup_{B_1(x)}|\Delta u|$, and we let $s=D^2$. Let $\varphi= \eta^2 v^{2q-1}$, where $q\ge 1$ and $\eta $ is some cut-off function in $C^\infty_0(B_1(x))$. By Lemma \ref{l:aprioriW2q_harmonic}, $\varphi\in W^{1,2}_0(B_1(x))$.  Noting that  $D\le v$ and $|\nabla u|\le v$ we have by H\"older inequality that
\begin{align} \label{B3}
&(2q-1)\int_{B_1(x)} v ^{2q-2}\eta^2|\nabla v |^2d\mu_g+2\int_{B_1(x)} v ^{2q-1}\eta\langle\nabla v ,\nabla \eta\rangle d\mu_g\notag\\
\le &\int_{B_1(x)}f\eta^2v ^{2q-1} d\mu_g+D\int_{B_1(x)}\left|\eta^2 v^{2q-3}\langle \nabla u,\nabla v\rangle\right| d\mu_g\notag\\
&+ 2D\int_{B_1(x)}\left|\langle \nabla u,\nabla\eta\rangle\eta v ^{2q-2}\right|d\mu_g+(2q-1)D\int_{B_1(x)}\left|\langle \nabla u,\nabla v \rangle\eta^2v ^{2q-3}\right|d\mu_g\notag\\ 
\le & \int_{B_1(x)}\eta^2 v^{2q}d\mu_g+2\int_{B_1(x)}|\nabla\eta|\eta v ^{2q} d\mu_g+2q\int_{B_1(x)} |\nabla v |\eta^2v ^{2q-1}d\mu_g. 
\end{align}

Denote $w=v ^{q}$, since $|\nabla u|\le  v$, by H\"older inequality and Sobolev inequality, we can estimate each term:
\begin{align}\label{B4}
(2q-1)\int_{B_1(x)} v ^{2q-2}\eta^2|\nabla v |^2d\mu_g=\frac{2q-1}{q^2}\int_{B_1(x)}\eta^2|\nabla w|^2d\mu_g,
\end{align}
and 
\begin{align}\label{B5}
2\int_{B_1(x)}| v ^{2q-1}\eta\langle\nabla v ,\nabla \eta\rangle| d\mu_g
&\le \frac{1}{4q}\int_{B_1(x)}\eta^2|\nabla w|^2d\mu_g+\frac{4}{q}\int_{B_1(x)}w^2|\nabla\eta|^2d\mu_g,
\end{align}
and 
\begin{align} \label{B7}
 \int_{B_1(x)}\left| \eta^2v ^{2q }\right|d\mu_g   \leq   \int_{B_1(x)}\eta^2 w^2d\mu_g,
\end{align}
\begin{align}\label{B6}
  2  \int_{B_1(x)}  |\nabla \eta|\eta v^{2q} d\mu_g \le  \int_{B_1(x)}|\nabla \eta|^2 w^2d\mu_g+ \int_{B_1(x)}\eta^2 w^2d\mu_g,
\end{align}
and
\begin{align} \label{B9}
2q \int_{B_1(x)}|\nabla v|\eta^2 v^{2q-1}d\mu_g
 \le \frac{1}{8q}  \int_{B_1(x)}|\nabla w|^2 \eta^2d\mu_g+8 q\int_{B_1(x)}\eta^2 w^2d\mu_g.
\end{align}

Apply (\ref{B4})--(\ref{B9}) to inequality (\ref{B3}), we get

\begin{align*}
\int_{B_1(x)}\eta^2|\nabla w|^2d\mu_g\leq \left(100q \right)^2\int_{B_1(x)}w^2(|\nabla \eta|^2+\eta^2)d\mu_g.
\end{align*}

Let $\eta=1$ on $B_r(x)$, $\eta=0$ on $M\backslash B_R(x)$, $\eta\le 1$ on $B_1(x)$ and $|\nabla \eta|\le \frac{C(g)}{(R-r)^2}$ on $B_1(x)$ .
Then we have
\begin{align} 
\int_{B_r(x)}|\nabla w|^2d\mu_g
\le
 \left(\frac{100q }{R-r}\right)^2\int_{B_1(x)}w^2d\mu_g.
\end{align}
Denote $\chi=\frac{n}{n-2}>1$, then by Sobolev inequality, we have
\begin{align} 
\left(\int_{B_r(x)}w^{2\chi}d\mu_g\right)^{\frac{1}{\chi}}
\le
\left(\frac{100q )}{R-r}\right)^2\int_{B_1(x)}w^2d\mu_g.
\end{align}
Thus we have
\begin{align} 
\left(\int_{B_r(x)}v^{2q\chi}d\mu_g\right)^{\frac{1}{2q\chi}}
\le
\left(\frac{100q }{R-r}\right)^{\frac{1}{q}}\left(\int_{B_1(x)}v^{2q}d\mu_g\right)^{\frac{1}{2q}}.
\end{align}
Take $q_i=\chi^i$, $r_i=\frac{1}{2}+\frac{1}{2^{i+1}}$, for $i=0,1,2,\ldots$, then we have
\begin{align}
\|v\|_{L^{2q_{i+1}}(B_{r_{i+1}}(x))}
\le \left(100(2\chi)^i \right)^{\frac{1}{\chi^i}}\|v\|_{L^{2q_i }(B_{r_{i }}(x))}.
\end{align}
By iteration we have
\begin{align}
\|v\|_{L^{2q_{i+1}}(B_{r_{j+1}}(x))}
\le (2\chi)^{\Sigma_{i=0}^j\frac{i}{\chi ^i}}\left( 100 \right)^{\Sigma_{i=0}^j\frac{1}{\chi^i}}\|v\|_{L^{2 }(B_1(x))}.
\end{align}
Since the seris $\Sigma_{i=0}^\infty\frac{i}{\chi ^i}$ and $\Sigma_{i=0}^\infty\frac{1}{\chi ^i}$ are both converge, we can let $i\to \infty$, and get
\begin{align}
\|v\|_{L^{\infty}(B_{\frac{1}{2}}(x))}
\le (2\chi)^{\Sigma_{i=0}^\infty\frac{i}{\chi ^i}}\left( 100 \right)^{\Sigma_{i=0}^\infty\frac{1}{\chi^i}}\|v\|_{L^{2 }(B_1(x))}\le C(n) ||v||_{L^2(B_1(x))}.
\end{align}
Thus we have
\begin{align}
\sup_{B_\frac{1}{2}(x)} |\nabla u|^2
\le C(n,g)\left(\int_{B_1(x)}|\nabla u|^2d\mu_g+\sup_{B_1(x)}|\Delta u|^2\right),
\end{align}
here and below $C(n,g)$ will denote a positive constant depending only on $n$, $\|g\|_{L^{\frac{n}{2}}}$ and the $L^2$ Sobolev constant in $B_1(x)$. Thus $|\nabla u|$ is bounded.

Moreover, if $\Delta u=0$, by rescaling we have 
	\begin{align}\label{e:gradient_byL2gradient}
		\sup_{B_{1/2}(x)}|\nabla u|^2\le C(n,g)\fint_{B_{3/4}(x)}|\nabla u|^2(y)d\mu_g(y).
	\end{align}

	Let $\varphi$ be a cut-off function satisfying $\varphi\equiv 1$ in $B_{3/4}(x)$ and ${\rm supp }\varphi\subset B_1(x)$ and $|\nabla \varphi|\le (n,g)$.  Multiplying $u\varphi^2$ to $\Delta u=0$ and integrating we have 
	\begin{align}
		0=\int_{B_1(x)}\varphi^2 u\Delta u d\mu_g(y)= -\int_{B_1(x)}\varphi^2 |\nabla u|^2d\mu_g(y)-\int_{B_1(x)}2\varphi u\langle \nabla \varphi,\nabla u\rangle d\mu_g(y).
	\end{align}
	Thus by Cauchy inequality, we get 
	\begin{align}
		\int_{B_1(x)}\varphi^2 |\nabla u|^2d\mu_g(y)\le 4\int_{B_1(x)}|\nabla \varphi|^2u^2d\mu_g(y)\le C(n,g)\int_{B_1(x)}u^2d\mu_g(y).
	\end{align}
	Combining with \eqref{e:gradient_byL2gradient}, we have
	\begin{align}
		\sup_{B_{1/2}(x)}|\nabla u|^2\le C(n,g)\fint_{B_1(x)}|\nabla u|^2d\mu_g(x).
	\end{align}
This finishes the proof of Proposition \ref{p:gradientdestimate_Deltabound}.
\end{proof}

\begin{proposition}[Heat kernel estimates]\label{p:heatkernel}
	Assume as Theorem \ref{t:RCDnonnegative}, then the heat kernel $\rho_t(x,y)$ of $(M^n,g)$ satisfies for $0<t\le 1$ that 
	\begin{itemize}
	   \item[(1)]  $\rho_t(x,y)=\rho_t(y,x)>0$ and $(\partial_t-\Delta_x)\rho_t(x,y)=0$ for all $t>0, x,y\in M$.
	   \item[(2)] $\lim_{t\to 0}\rho_t(x,y)=\delta_x(y)$.
		\item[(3)]  $\rho_t(x,y) \le Ct^{-n/2}e^{-d^2(x,y)/C}$ for all $x,y\in M$.
		\item[(4)]  $|\nabla_x\rho_t(x,y)|\le Ct^{-(n+1)/2}e^{-d^2(x,y)/C}$ for all $x,y\in M$.
	\end{itemize}
	where the constant $C$ depends on $(M^n,g)$. 
\end{proposition}
\begin{proof}
	(1) and (2) are the basic properties of heat kernel. To see (3), noting that the metric is $C^0$ and asymptotically flat,  the upper bound (3) follows directly by Theorem 0.2 of \cite{Sturm}. Noting the gradient estimates for harmonic functions in Proposition \ref{p:gradientdestimate_Deltabound},  (4) is a consequence of Theorem 1.2 of \cite{CJKS}.  Hence we complete the proof. Actually, one can also argue as the gradient estimates of harmonic functions to get the heat kernel gradient estimates. 
\end{proof}
\vskip 3mm

\begin{lemma}[$W^{1,2}$-approximation]\label{l:w12approximateheatkernel}
	Assume as Theorem \ref{t:RCDnonnegative} and let $u\in D(\Delta)$ and $\Delta u\in W^{1,2}$. Define $u_t(x)=\int_M\rho_t(x,y)u(y)dy$. Then 
	\begin{enumerate}
		\item $u_t\to u$ in $W^{1,2}$-sense on any compact subset as $t\to 0$.
		\item $\Delta u_t\to \Delta u$ in $L^2$-sense on any compact subset as $t\to 0$.
		\item For any $t>0$,  $|\nabla u_t|,|\Delta u_t|, |\nabla \Delta u_t|$ are bounded.
		\item For any $t>0$, $u_t\in W^{2,p}$ when $n\le p<\infty$, and $u_t\in W^{2,q}$ for any $q<\infty$ when $p=\infty$.
	\end{enumerate}
	 
\end{lemma}
\begin{proof}
	Let us first show that $u_t\to u$ in $L^2$-sense. Since we can use continuous functions approximating $u$, for simplicity, let us assume $u$ is continuous. Then for any compact measurable subset $\Omega\subset M$ we have
	\begin{align}
		\int_{\Omega} |u_t(x)-u(x)|^2d\mu_g(x)&= \int_{\Omega}\left(\int_M u(y)\rho_t(x,y)d\mu_g(y)-u(x)\right)^2d\mu_g(x)\\
		&\le \int_{\Omega}\int_M |u(y)-u(x)|^2\rho_t(x,y)d\mu_g(y)d\mu_g(x).
	\end{align}
	By heat kernel upper bound estimate in Proposition \ref{p:heatkernel} and noting that $u$ is continuous we conclude that $\lim_{t\to 0}\int_{\Omega} |u_t(x)-u(x)|^2d\mu_g(x)=0$. Then, noting that $u_t(x)=\int_Mu(y)\rho_t(x,y)d\mu_g(y)$, we have 
	\begin{align}
		\Delta u_t(x)=\int_M u(y)\Delta_x\rho_t(x,y)d\mu_g(y)=\int_M u(y)\partial_t\rho_t(x,y)d\mu_g(y) =\int_Mu(y)\Delta_y\rho_t(x,y)d\mu_g(y)=\int_M\Delta u(y)\rho_t(x,y)d\mu_g(y). 
	\end{align}
	Therefore, the same argument as above shows that $\Delta u_t\to \Delta u$ in $L^2$-sense on any compact measurable subset, which proves (2). Let us now show $u_t\to u$ in $W^{1,2}$-sense. Actually, for any compact support cut-off function $\varphi$ we can compute 
	\begin{align}
		\int_M \varphi^2(x) |\nabla u_t-\nabla u|^2(x)d\mu_g(x)\le& \int_M |\Delta u_t-\Delta u|\cdot |u_t-u|\varphi^2(x)d\mu_g(x)+\int_M 2|\nabla \varphi|(x)|\nabla u_t-\nabla u|(x) |u_t-u|(x)\varphi(x)d\mu_g(x)\\
		\le& \frac{1}{2}\int_M \varphi^2(x) |\nabla u_t-\nabla u|^2(x)d\mu_g(x)+2\int_M |\nabla \varphi|^2 |u_t-u|^2(x)d\mu_g(x)\\
		&+\left(\int_M|\Delta u_t-\Delta u|^2\varphi^2(x)d\mu_g(x)\right)^{1/2}\left(\int_M|u_t- u|^2\varphi^2(x)d\mu_g(x)\right)^{1/2}.
	\end{align}
	Letting $t\to 0$, we get $\lim_{t\to 0}\int_M \varphi^2(x) |\nabla u_t-\nabla u|^2(x)d\mu_g(x)=0$. Hence we complete the proof of (1) and (2).

	(3) follows directly from the heat kernel gradient estimate in Proposition \ref{p:heatkernel}. By the gradient bound in (3), we get by Remark \ref{r:W2qestimate} that (4) holds. Hence we finish the whole proof. 
\end{proof}

\vskip 3mm

\begin{proposition}[Weak Bochner inequality]\label{p:weakbochnerinequaliyt}
	Assume as Theorem \ref{t:RCDnonnegative}. Then for any $u\in D(\Delta)$ such that $\Delta u\in W^{1,2}$, and any nonnegative bounded test function $\varphi\in D(\Delta)$ with $|\Delta \varphi|\in L^\infty$, the following Bochner inequality holds
	\begin{align}\label{e:weakbochnoerlowerbound}
		\frac{1}{2}\int_M \Delta \varphi |\nabla u|^2 d\mu_g-\int_M \varphi \langle \nabla \Delta u,\nabla u\rangle d\mu_g\ge \frac{1}{n}\int_M\varphi (\Delta u)^2d\mu_g. 
	\end{align}	
\end{proposition}
\begin{proof}
Let $\varphi$ be a nonnegative bounded test function, $\varphi\in D(\Delta)$ with $|\Delta \varphi|\in L^\infty$. By Proposition \ref{p:gradientdestimate_Deltabound}, we have that $|\nabla \varphi|$ is bounded. For any given $R>0$, assume $\eta_R$ is a smooth cut-off function satisfying $\eta_R\equiv 1$ on $B_R(\Sigma)$  and $\eta_R\equiv 0$ outside $B_{2R}(\Sigma)$. Let $\varphi_R=\varphi \eta_R$. Since the metric $g$ is smooth away from $\Sigma$ and $|\nabla \eta_R|$ vanishes on $B_R(\Sigma)$, then one can check directly that $\varphi_R\in W^{1,2}$ and $\Delta \varphi_R$ is bounded.  Let us now first prove \eqref{e:weakbochnoerlowerbound} for $\varphi_R$. Then letting $R\to \infty$, we will get the desired formula by noting $\Delta \varphi_R\to \Delta \varphi$ in $L^\infty$ sense and $\varphi_R\to \varphi$ in $L^\infty$ sense in any compact subset.

	Let $u_t(x)=\int_Mu(y)\rho_t(x,y)dy$. By Lemma \ref{l:w12approximateheatkernel} we have $u_t\in W^{2,q}_{loc}$ for any $q<p$. For any fixed $t>0$, assume $h_{i,t}$ is smooth and converges in $W^{2,q}$-sense to $u_t$ as $i\to \infty$. In particular, $|\nabla h_{i,t}|\in L^\infty$ and $|\nabla^2h_{i,t}|\in L^p$ for any $n\le p\le \infty$. Here we use $h_{i,t} $ approximating $u_t$ again just because $|\nabla^2u_t|$ may not in $L^p$ when $p=\infty$. For $n\le p<\infty$, one can only use $u_t$ in the following argument but not $h_{i,t}$.

 By Bochner formula, we have on $M\setminus \Sigma$ that 
	\begin{align}\label{e:bochner_ut}
		\frac{1}{2}\Delta |\nabla h_{i,t}|^2&=|\nabla^2h_{i,t}|^2+\langle \nabla \Delta h_{i,t},\nabla h_{i,t}\rangle +\Ric(\nabla h_{i,t},\nabla h_{i,t})\notag\\
		&=|\nabla^2h_{i,t}|^2+\langle \nabla \Delta h_{i,t},\nabla h_{i,t}\rangle\notag\\
		&\ge \frac{1}{n}(\Delta h_{i,t})^2+\langle \nabla \Delta h_{i,t},\nabla h_{i,t}\rangle.
	\end{align}
	For any $\epsilon>0$, assume $\psi_\epsilon$ is a cut-off function of $\Sigma$ from Lemma \ref{l:cut-off} satisfying $\psi_\epsilon\equiv 1$ on $M\setminus B_\epsilon(\Sigma)$ and $\psi_\epsilon$ vanishes in a neighborhood of $\Sigma$, and $\lim_{\epsilon\to 0}\int_M|\nabla \psi_\epsilon|^qd\mu_g=0$ with $q=\frac{p}{p-1}$. Multiplying $\varphi_R\psi_\epsilon$ to both sides of \eqref{e:bochner_ut} and integrating by parts, we get 
	\begin{align}\label{e:bochnerutvarphiRepsilon}
		-\frac{1}{2}\int_M\langle \nabla (\varphi_R\psi_\epsilon),\nabla |\nabla h_{i,t}|^2\rangle d\mu_g &\ge \frac{1}{n}\int_M(\varphi_R\psi_\epsilon)(\Delta h_{i,t})^2 d\mu_g+\int_M \left(-(\Delta h_{i,t})^2(\varphi_R\psi_\epsilon)-\Delta h_{i,t}\langle \nabla h_{i,t},\nabla (\varphi_R\psi_\epsilon)\rangle \right) d\mu_g\\
		&:=I+II
	\end{align}
	Let us consider the left hand side. 
	\begin{align}
		\int_M\langle \nabla (\varphi_R\psi_\epsilon),\nabla |\nabla h_{i,t}|^2\rangle d\mu_g&=\int_M \psi_\epsilon \langle\nabla \varphi_R,\nabla |\nabla h_{i,t}|^2\rangle d\mu_g +\int_M\varphi_R\langle \nabla \psi_\epsilon,\nabla |\nabla h_{i,t}|^2\rangle  d\mu_g\\
		&=-\int_M\psi_\epsilon \Delta \varphi_R |\nabla h_{i,t}|^2 d\mu_g-\int_M|\nabla h_{i,t}|^2\langle \nabla \psi_\epsilon,\nabla \varphi_R\rangle d\mu_g+\int_M\varphi_R\langle \nabla \psi_\epsilon,\nabla |\nabla h_{i,t}|^2\rangle d\mu_g. 
	\end{align}
	Noting that $h_{i,t}$ is smooth and $|\nabla h_{i,t}|, |\nabla \varphi_R|$ are bounded and $|\nabla^2h_{i,t}|\in L^p$, letting $\epsilon\to 0$, the last two terms converge to zero. Thus we get 
	\begin{align}
		\lim_{\epsilon\to 0}\int_M\langle \nabla (\varphi_R\psi_\epsilon),\nabla |\nabla h_{i,t}|^2\rangle d\mu_g=-\int_M \Delta \varphi_R |\nabla h_{i,t}|^2d\mu_g
	\end{align}
In \eqref{e:bochnerutvarphiRepsilon}, when $\epsilon\to 0$, the term $I$ converges to $\frac{1}{n}\int_M\varphi_R(\Delta h_{i,t})^2$. Estimating similarly as above, letting $\epsilon\to 0$, the term $II$ would converge to $\int_M \left(-(\Delta h_{i,t})^2\varphi_R-\Delta h_{i,t}\langle \nabla h_{i,t},\nabla \varphi_R\rangle \right) d\mu_g$. Therefore, we arrive at 
\begin{align}
	\frac{1}{2}\int_M \Delta \varphi_R |\nabla h_{i,t}|^2 d\mu_g\ge \frac{1}{n}\int_M\varphi_R(\Delta h_{i,t})^2 d\mu_g+\int_M \left(-(\Delta h_{i,t})^2\varphi_R-\Delta h_{i,t}\langle \nabla h_{i,t},\nabla \varphi_R\rangle \right)d\mu_g.
\end{align}
Since $h_{i,t}\to u_t$ in $W^{2,q}$-sense for some $5/2\le q<n$, letting $i\to \infty$, we get 
	\begin{align}
	\frac{1}{2}\int_M \Delta \varphi_R |\nabla u_{t}|^2 d\mu_g\ge \frac{1}{n}\int_M\varphi_R(\Delta u_{t})^2d\mu_g+\int_M \left(-(\Delta u_{t})^2\varphi_R-\Delta u_{t}\langle \nabla u_{t},\nabla \varphi_R\rangle \right)d\mu_g.
\end{align}
	Letting $t\to 0$, by the approximating Lemma \ref{l:w12approximateheatkernel} we get 

	\begin{align}
		\frac{1}{2}\int_M \Delta \varphi_R |\nabla u|^2 d\mu_g&\ge \frac{1}{n}\int_M\varphi_R(\Delta u)^2 d\mu_g+\int_M \left(-(\Delta u)^2\varphi_R  -\Delta u\langle \nabla u,\nabla \varphi_R\rangle \right) d\mu_g\\
		&=\frac{1}{n}\int_M\varphi_R(\Delta u)^2 d\mu_g+\int_M \varphi_R\langle \nabla \Delta u,\nabla u\rangle d\mu_g,
	\end{align}
	where the last equality follows from the definition of Dirichlet Laplacian. Letting $R\to \infty$, we finish the proof. 
	\end{proof}

Now we are ready to prove Theorem \ref{t:RCDnonnegative}. 
\begin{proof}[Proof of Theorem \ref{t:RCDnonnegative} ]
	Noting the Bochner inequality proved in Proposition \ref{p:weakbochnerinequaliyt} and Theorem 6 of \cite{EKS15}(see also \cite{AMS}), $(M^n,g)$ with Lebesgue measure is an RCD space with nonnegative Ricci curvature, see also \cite{BKMR}. 
\end{proof}

Now we can finish the proof of Theorem \ref{thm1.2}.
\begin{proof}[Proof of rigidity part of Theorem \ref{thm1.2}]
Now the rigidity part of Theorem \ref{thm1.2} follows immediately from Theorem \ref{t:RCDnonnegative} and the volume stability theorem of RCD space (\cite{LoVi,Stru06}, see also Theorem 1.6 of \cite{DeGi}). Actually, by Lemma \ref{lm6.6} and Theorem \ref{t:RCDnonnegative}, the manifold $(M^n,g)$ has nonnegative Ricci curvature in RCD sense.

take any point $x\in M$, noting that the manifold is asymptotically flat, we have 
\begin{align}
	\lim_{R\to \infty}\frac{\Vol(B_R(x))}{\omega_n R^n}=1,
\end{align}
By volume comparison of RCD space with nonnegative Ricci curvature, we have 
\begin{align}
	\Vol(B_R(x))=\omega_n  R^n.
\end{align}
By volume rigidity Theorem 1.6 of \cite{DeGi} (or Corollary 1.7 of \cite{DeGi}), we get $B_R(x)$ is isometric to $B_R(0^n)\subset \mathbb{R}^n$. This implies $M$ is isometric to $\mathbb{R}^n$. 
\end{proof}

\appendix

\section{ Cut-off functions}
In this appendix, we will prove the following cut-off function lemma. The result is well know to the experts.  For the sake of convenience, we will give a proof here following the argument from Cheeger \cite{Cheeger}.
\begin{lemma}\label{l:cut-off}
	Let $(M^n,h)$ be a smooth manifold. Assume $\Sigma\subset M$ is a closed subset. Then there exists a sequence of cut-off function $\varphi_\epsilon$ of $\Sigma$ such that the following holds:
	\begin{itemize}
		\item[(1)] $0\le \varphi_\epsilon\le 1$ and $\varphi\equiv 0$ in a neighborhood of $\Sigma$ and $\varphi_\epsilon \equiv 1 $ on $M\setminus B_{\epsilon}(\Sigma)$.
		\item[(2)] If $\cH^{n-1}(\Sigma)=0$, then $\lim_{\epsilon\to 0}\int_M|\nabla \varphi_\epsilon|(x)dx=0.$
		\item[(3)] If $\cH^{n-p}(\Sigma)<\infty$ with $p>1$, then  $\lim_{\epsilon\to 0}\int_M|\nabla \varphi_\epsilon|^{p}(x)dx=0.$
	\end{itemize}
\end{lemma}
\begin{proof}
	If $\cH^{n-1}(\Sigma)=0$, by definition, for any $\delta>0$ there exists a covering $\{B_{r_i}(x_i),x_i\in \Sigma\}_{i=1}^{N_\delta}$ of $\Sigma$ such that $r_i\le \delta$ and $\sum_{i=1}^{N_\delta} r_i^{n-1}\le \delta$. Define smooth $\eta_i$ such that 
	\begin{itemize}
		\item[(a)] $\eta_i\equiv 0$ in $B_{r_i}(x_i)$ and $\eta_i\equiv 1$ in $M\setminus B_{2r_i}(x_i)$.
		\item[(b)] $|\nabla \eta_i|\le 10 r_i^{-1}$.  
	\end{itemize}  
	Let us now define $\psi_\delta(x)=\min_{i=1,\cdots,N_\delta}\eta_i(x)$. Then $\psi_\delta\equiv 0$ in a neighborhood of $\Sigma$ and $\psi_{\delta}\equiv 1$ on $M\setminus \cup_{i}B_{2r_i}(x_i)$. Noting that $r_i\le \delta$ and $x_i\in \Sigma$, we have $\psi_{\delta}\equiv 1$ on $M\setminus B_{2\delta}(\Sigma)$. By definition, we have $|\nabla\psi_\delta(p)|\le \Sigma_{i=1}^{N_\delta}|\nabla \eta_i(p)|$ for any point $p$ such that $\nabla\psi_\delta(p)$ exists. Since $\psi_\delta$ is Lipschitz, we have
	\begin{align*}
		 \int_M|\nabla \psi_\delta|(x)dx&\leq\Sigma_{i=1}^{N_\delta} \int_M|\nabla \eta_i|(x)dx \\
		&\leq\Sigma_{i=1}^{N_\delta}C(n)r_i^{n-1}\\
		&\leq C(n)\delta. 
	\end{align*}
Thus we have
	 \begin{align*}
		\lim_{\delta\to 0}\int_M|\nabla \psi_\delta|(x)dx=0. 
	\end{align*}
Hence $\varphi_\epsilon=\psi_\delta$ with $\delta\le \epsilon/10$ satisfies the lemma. \\

If $\cH^{n-p}(\Sigma)\le A<\infty$ with $p>1$, by definition, there exists a covering $\{B_{r_{i,1}}(x_{i,1}),x_{i,1}\in \Sigma,~~r_{i,1}\le \delta\}_{i=1}^{N_{\delta,1}}$ of $\Sigma$ such that 
\begin{align}
	\sum_{i=1}^{N_{\delta,1}}r_{i,1}^{n-p}\le 2A.
\end{align}
Denote $\delta_2=\min\{2^{-1}, ~\min\{r_{i,1},i=1,\cdots, N_{\delta,1}\}/10\} >0$. There exists a covering $\{B_{r_{i,2}}(x_{i,2}),x_{i,2}\in \Sigma,~~r_{i,2}\le \delta_2\}_{i=1}^{N_{\delta_2,2}}$ of $\Sigma$ such that 
\begin{align}
	\sum_{i=1}^{N_{\delta_2,2}}r_{i,2}^{n-p}\le 2A.
\end{align}
Inductively, for any integer $\alpha\ge 2$  there exists 
a covering $\{B_{r_{i,2}}(x_{i,\alpha}),x_{i,\alpha}\in \Sigma,~~r_{i,\alpha}\le \delta_\alpha\}_{i=1}^{N_{\delta_\alpha,\alpha}}$ of $\Sigma$ such that 
\begin{align}
	\sum_{i=1}^{N_{\delta_\alpha,\alpha}}r_{i,\alpha}^{n-p}\le 2A,
\end{align}
and $\delta_\alpha\le \min\{\alpha^{-1}, ~\min\{r_{i,{\alpha-1}},i=1,\cdots, N_{\delta_{\alpha-1},\alpha-1}\}/10\}>0$. Let us now construct the desired cut-off function. For each $B_{r_{i,\alpha}}(x_{i,\alpha})$ we define $\eta_{i,\alpha}$ such that $\eta_{i,\alpha}\equiv 0$ in $B_{r_{i,\alpha}}(x_{i,\alpha})$ and $\eta_{i,\alpha}\equiv 1$ in $M\setminus B_{2r_{i,\alpha}}(x_{i,\alpha})$, moreover $|\nabla \eta_{i,\alpha}|\le 10 r_{i,\alpha}^{-1}$.  Define $\psi_\alpha(x)=\min\{\eta_{i,\alpha}(x)~~|~~i=1,\cdots, N_{\delta_{\alpha},\alpha}\}$.  Define for any integer $k$ 
\begin{align}
	\phi_k=\frac{1}{2^k}\sum_{\alpha=2^{k}+1}^{2^{k+1}} \psi_\alpha(x). 
\end{align}
Then $\varphi_\epsilon=\phi_k$ satisfies the lemma when $\epsilon^{-1}\le k< \epsilon^{-1}+1$. 
\end{proof} 

\section{Gradient estimates}
If the metric has bounded $W^{1,p}$-norm with $p>n$, by elliptic estimate we can get the following gradient estimate. The idea of this estimate is the same as gradient estimate on K\"ahler manifolds with $L^p$ Ricci potential, see \cite{TiZh}.
\begin{lemma}\label{l:gradient_W1pmetric}
	Let $(M^n,g)$ be smooth and $\Delta u=f$ on $B_r(x)$ with $u,f$ smooth. Then 
	\begin{align}
		\sup_{B_{r/2}(x)}|\nabla u|^2\le C(n,g)\left(\fint_{B_r(x)}|\nabla u|^2d\mu_g+r^2\sup_{B_r(x)}|f|^2\right),
	\end{align}
	where $C(n,g)$ depends only on the $W^{1,p}$-norm of $g$ on $B_r(x)$ for given $p>n$.
\end{lemma}
\begin{proof}
Since the estimate is scaling invariant and noting that the estimate is local, without loss of generality let us assume $r=1$ and $B_1(x)$ is in a coordinate neighborhood.  We denote $v=   \sqrt{|\nabla u|^2+s}$, $s>0$. By Bochner formula, we have 
\begin{align*}
\Delta v^2=\Delta |\nabla u|^2&=2|\nabla ^2u|^2+2\langle \nabla \Delta u,\nabla u\rangle+2\Ric(\nabla u,\nabla u)\\
&\ge 2|\nabla ^2u|^2+2\langle \nabla \Delta u,\nabla u\rangle+2\Ric(\nabla u,\nabla u).
\end{align*}
Let $\varphi$ be a function in $C^\infty_0(M)$, we have
\begin{align*}
\varphi \Delta v\ge \frac{\varphi}{2v}\Delta v^2\ge \varphi v^{-1}|\nabla ^2u|^2+ \varphi v^{-1}\langle \nabla \Delta u,\nabla u\rangle+\varphi v^{-1}\Ric(\nabla u,\nabla u).
\end{align*}
In a coordinate neighborhood, and we can write $\Ric=\partial^2 g+\partial g*\partial g$. Thus by integration by part, we have
\begin{align*}
-\int_{B_1(x)}\langle \nabla\varphi,\nabla v\rangle d\mu_g
&\ge -\int_{B_1(x)}\varphi v^{-1}(\Delta u)^2d\mu_g +\int_{B_1(x)}\varphi v^{-2}\Delta u \langle\nabla u,\nabla v\rangle d\mu_g-\int_{B_1(x)}v^{-1}\Delta u\langle\nabla u,\nabla \varphi\rangle d\mu_g\\
&+ \int_{B_1(x)}\varphi v^{-1}\partial g*\partial g*\nabla u*\nabla ud\mu_g+ \int_{B_1(x)}\varphi v^{-1}\partial^2 g*\nabla u*\nabla ud\mu_g+ \int_{B_1(x)}\varphi v^{-1}|\nabla ^2u|^2d\mu_g.
\end{align*}
Thus we have
\begin{align*}
\int_{B_1(x)}\langle \nabla\varphi,\nabla v\rangle d\mu_g
&\le \int_{B_1(x)}\varphi v^{-1}(\Delta u)^2d\mu_g -\int_{B_1(x)}\varphi v^{-2}\Delta u \langle\nabla u,\nabla v\rangle d\mu_g+\int_{B_1(x)}v^{-1}\Delta u\langle\nabla u,\nabla \varphi\rangle d\mu_g\\
&-\int_{B_1(x)}\varphi v^{-1}\partial g*\partial g*\nabla u*\nabla ud\mu_g+ \int_{B_1(x)}\varphi v^{-1}|\nabla ^2u|\partial g*\nabla u d\mu_g-\int_{B_1(x)}\varphi v^{-1}|\nabla ^2u|^2d\mu_g\\
& +\int_{B_1(x)}\varphi v^{-2}\partial g*\nabla v*\nabla u*\nabla ud\mu_g+\int_{B_1(x)}v^{-1}\partial g*\nabla \varphi*\nabla u*\nabla ud\mu_g.
\end{align*}

Since 
\[\int_{B_1(x)}\varphi v^{-1}|\nabla ^2u|\partial g*\nabla u d\mu_g\le \frac{1}{2}\int_{B_1(x)}\varphi v^{-1}|\nabla ^2u|^2d\mu_g+C(n,g)\int_{B_1(x)}\varphi v^{-1}|\partial g|^2 |\nabla u|^2d\mu_g,\]
we have 
\begin{align*}
\int_{B_1(x)}\langle \nabla\varphi,\nabla v\rangle d\mu_g
&\le \int_{B_1(x)}\varphi v^{-1}(\Delta u)^2d\mu_g -\int_{B_1(x)}\varphi v^{-2}\Delta u \langle\nabla u,\nabla v\rangle d\mu_g+\int_{B_1(x)}v^{-1}\Delta u\langle\nabla u,\nabla \varphi\rangle d\mu_g\\
&-\int_{B_1(x)}\varphi v^{-1}\partial g*\partial g*\nabla u*\nabla ud\mu_g+ C(n,g)\int_{B_1(x)}\varphi v^{-1}|\partial g|^2 |\nabla u|^2d\mu_g\\
& +\int_{B_1(x)}\varphi v^{-2}\partial g*\nabla v*\nabla u*\nabla ud\mu_g+\int_{B_1(x)}v^{-1}\partial g*\nabla \varphi*\nabla u*\nabla ud\mu_g,
\end{align*}
here and below $C(n,g)$ will denote a positive constant depends only on the $W^{1,p}$-norm of $g$, and can vary from line to line.

If $\Delta u\equiv 0$, we will let $s\to 0^{+}$ in the following arguement.

If $\Delta u\not\equiv 0$, we denote $D=\sup_{B_1(x)}|\Delta u|$, and we let $s=D^2$. Since $D\le v$ and $|\nabla u|\le v$, we have

\begin{align*}
\int_{B_1(x)}\langle \nabla\varphi,\nabla v\rangle d\mu_g
&\le\int_{B_1(x)}\varphi v d\mu_g+\int_{B_1(x)}\varphi| \nabla v | d\mu_g+\int_{B_1(x)}v|\nabla \varphi|d\mu_g\\
&+C(n,g)\int_M \varphi v|\partial g|^2d\mu_g+C(n,g)\int_{B_1(x)}\varphi v|\partial g|d\mu_g\\
&+C(n,g)\int_{B_1(x)}\varphi|\nabla v||\partial g|d\mu_g+C(n,g)\int_{B_1(x)}|\partial g||\nabla \varphi|d\mu_g.
\end{align*}

By choosing test function $\varphi= \eta^2 v^{2q-1}$, where $q\ge 1$ and $\eta $ is some cut-off function in $C^\infty_0(B_1(x))$, we can argue completely as 
 the proof Proposition \ref{p:gradientdestimate_Deltabound} to get the desired estimate:
\begin{align}
\sup_{B_\frac{r}{2}(x)} |\nabla u|^2
\le C(n,g) \left(\fint_{B_r(x)}|\nabla u|^2d\mu_g+r^2\sup_{B_r(x)} |f|^2\right),
\end{align}
where $C(n,g)$ depends only on the $W^{1,p}$-norm of $g$ on $B_r(x)$ for given $p>n$. This completes the proof.
\end{proof}

\bibliographystyle{plain}

\end{document}